\documentclass[a4paper, twoside,12pt,]{article}

\usepackage{amsmath,amssymb,amsthm,braket,nccmath, mathrsfs}
\usepackage[margin=30mm]{geometry}
\usepackage{indentfirst}
\usepackage{fancyhdr}
\usepackage{comment}

\numberwithin{equation}{section}

\newcommand{\N}{\mathbb{N}}
\newcommand{\R}{\mathbb{R}}
\newcommand{\Z}{\mathbb{Z}}
\newcommand{\ep}{\varepsilon}
\newcommand{\bfx}{\mathbf{x}}
\newcommand{\f}{\frac}
\newcommand{\lmd}{\lambda}
\newcommand{\cF}{\mathcal{F}}
\newcommand{\cH}{\mathcal{H}}

\newcommand{\psiep}{\psi^\ep}
\newcommand{\phiep}{\phi^\ep}

\newcommand{\tpsi}{\tilde{\psi}^\ep}
\newcommand{\Fav}{F_{\text{av}}}
\newcommand{\h}{\f{1}{2}}
\newcommand{\om}{\omega}
\newcommand{\Tmax}{T_{\text{max}}}
\newcommand{\beq}{\begin{equation}}
\newcommand{\eeq}{\end{equation}}
\newcommand{\bsp}{\begin{equation}\begin{split}}
\newcommand{\esp}{\end{split}\end{equation}}
\newcommand{\al}{\alpha}

\theoremstyle{theorem}

\newtheorem*{theorem*}{Theorem}
\newtheorem{thm}{Theorem}[section]

\newtheorem*{definition*}{Definition}
\newtheorem{lemma}[thm]{Lemma}
\newtheorem*{lemma*}{Lemma}
\newtheorem{prop}[thm]{Proposition}
\newtheorem{remark}[thm]{Remark}

\title{Averaging of strong magnetic nonlinear Schr\"{o}dinger equations in the energy space}
\author{\small{JUMPEI KAWAKAMI}\thanks{Department of mathematical Sciences, Kyoto University, Kyoto 606-8502 JAPAN. 	
		E-mail: jumpeik@kurims.kyoto-u.ac.jp}}{}
\date{}
\begin{document}

\maketitle

\markboth{\centerline{\footnotesize{JUMPEI KAWAKAMI}}}{\centerline{\footnotesize{Averaging of strong magnetic NLSs in energy space}}}

\renewcommand{\thefootnote}{\fnsymbol{footnote}}
\renewcommand{\thefootnote}{\arabic{footnote}}

\begin{abstract}
In this study, we consider two nonlinear Schr\"{o}dinger-type models that are derived by R L.~Frank, F.~M\'{e}hats, C.~Sparber (2017) to study 3D nonlinear Schr\"{o}dinger equations under strong magnetic fields. One model is derived by spatial scaling and the other is obtained by averaging the spatial scaled model over time. We study these models in the energy space to obtain global solutions and improve the convergence result over an arbitrarily long time. Regarding the nonic nonlinear power of the time averaged model,  we prove  a scattering result under a scaling-invariant small-energy condition, which underlines energy-criticality of the nonic case.
\end{abstract}

\section{Introduction}
\label{ntro}
\subsection{Derivation of the NLS-type models}

 In this paper, we study two nonlinear Schr\"{o}dinger-type models  with strong magnetic confinement. These models were introduced in \cite{2}  to study the asymptotic scaling limit of nonlinear Schr\"{o}dinger equations (NLS) under strong magnetic fields, which appear in \cite{BEC}. We will follow the derivation in \cite[Section 1]{2}. First, we consider the following 3-dimensional NLS-type model:
\begin{equation}
\label{eq1.1}
	i\partial_{t}\psi= \frac{1}{2}(-i\nabla_{\bfx}+A^{\ep}(\bfx))^2\psi +V(z)\psi +\beta^{\ep} |\psi|^{2\sigma}\psi,
\end{equation}
where $(t,\mathbf{x})\in \R \times \R^3 $, $\sigma \in \mathbb{N} $, and $\beta^\ep \in \mathbb{R}$.  We will denote the spatial variables by $\bfx=(x_1,x_2,z)$ and state $x=(x_1,x_2)$. The real-valued potential $V$ is assumed to be smooth and sub-quadratic:
\begin{equation}
\label{V}
 \f{d^\al}{dz^\al}V(z) \in L^\infty(\R) \qquad \text{for any}\quad  \al\ge2.
\end{equation}
For example, $V(z)=\pm z^2$.
The vector potential $A^\ep$ is given by 
\[ A^{\ep}(\mathbf{x})=\frac{1}{2\ep^2}(-x_2,x_1,0), \]
where $0<\ep \ll 1$ is a small parameter. Because $A^{\ep}$ is divergence free, $\nabla_{\mathbf{x}} \cdot A^{\ep}=0$,
\[ (-i \nabla_{\bfx} +A^{\ep}(\bfx))^2=-\Delta_{\mathbf{x}}+\frac{1}{4\ep^4}|x|^2-\frac{i}{\ep^2}(x_1 \partial_{x_2}-x_2 \partial_{x_1}) .\]
The corresponding magnetic field is 
\[ B^{\ep}=\nabla \times A^{\ep}=\frac{1}{\ep^2}(0,0,1) \in \R^3.\]
This is a constant magnetic field in the $z$-direction with field strength $|B^\ep|=\f{1}{\ep^2} \gg 1$.

The objective of \cite{2} (which is also our objective) was to  analyze ``the strong magnetic confinement limit'' as $\ep \to +0$. Then, the initial data for equation (\ref{eq1.1}) is assumed of the form 
\[ \psi(0,\bfx)=\ep^{-1} \psi_0 (\frac{x}{\ep},z). \]
This assumption implies that the initial wave function is already confined at the scale $\ep$ in the $x$-directions. Accordingly, we rescale
\[  x'=\f{x}{\ep},  \quad  z'=z, \quad  \psi^\ep (t,x',z')=\ep \psi(t,\ep x', z'). \]
and  let $\lmd \in \R$ be a fixed constant. Finally, we rewrite $\beta^\ep=\lmd \ep^{2\sigma} \ll 1$. 

Thorough these procedures, equation (\ref{eq1.1}) becomes 
\begin{equation}
\label{ep}
i\partial_t \psiep=\f{1}{\ep^2}\cH \psiep+\cH_z \psiep +\lmd |\psiep|^{2\sigma} \psiep \tag{$\ep$-NLS}
\end{equation}
where, using $x^\perp=(-x_2,x_1)$ we denote
\[\cH:=\f{1}{2}(-i\nabla_x+\f{1}{2}x^\perp)^2=-\f{1}{2}\Delta_x+\f{1}{8}|x|^2-\f{i}{2}x^\perp\cdot\nabla_x , \]
and we state
\[ \cH_z:=-\f{1}{2}\partial_z^2+V(z). \]
Solution to (\ref{ep}) conserves mass $M$ and Hamiltonian $E^\ep$, which are defined as
\begin{equation}\label{mass}
M[\psi]:= \int_{\R^3} |\psi(\bfx)|^2 d\bfx
\end{equation}
\begin{equation}\label{eneggy}
\begin{split}
E^\ep[\psi]&:=\f{1}{2}\Big(\f{1}{2\ep^2}\int_{\R^3} |\nabla_x \psi|^2 d\bfx + \f{1}{8\ep^2} \int_{\R^3} |x\psi|^2  d\bfx-\f{i}{2\ep^2}\int_{\R^3} (x^\perp\cdot\nabla_x\psi)\overline{\psi} d\bfx\\
&\quad + \h \int_{\R^3} |\partial_z\psi|^2 d\bfx+\int_{\R^3}V(z)|\psi(\bfx)|^2 d\bfx\Big) + \f{\lmd}{2(\sigma+1)} \int_{\R^3} |\psi|^{2\sigma+2} d\bfx.
\end{split} 
\end{equation}

The operator $\cH$ is essentially self-adjoint on $C_0^\infty(\R^2) \subset L^2(\R^2)$ with the pure point spectrum given by
\begin{equation}\label{specn}
\text{spec}\cH=\{n+\h : n \in \N_0\}.
\end{equation}

Based on (\ref{ep}), $\cH$ induces high frequency oscillations in time $\propto\mathcal{O}(\ep^{-2})$ within the solution $\psiep$. By filtering these oscillations, we expect to observe the following limit 
\begin{equation}\label{smc}
 \phiep(t,\bfx):= e^{it\cH/\ep^2}\psiep(t,\bfx)\longrightarrow\phi(t,\bfx) \quad \text{as}\quad \ep \to +0
\end{equation}
in a strong norm. To describe the behavior of the limit $\phi$,  \cite{2} used the space (as a natural  functional framework),
\[ \Sigma^2=\{u \in H^2(\R^3): |\bfx|^2u \in L^2(\R^3)\}\]
\[ \| u\|_{\Sigma^2}:=(\| u \|_{H^2}^2+\| |\bfx|^2u\|_{L^2}^2)^{1/2},  \]
which is a Banach algebra. First, we introduce the following nonlinear function,
\begin{equation}\label{Fnl}
\begin{split}
F(\theta,u)&:=e^{i\theta\cH}(|e^{-i\theta\cH}u|^{2\sigma} e^{-i\theta\cH}u)\\
&=e^{i\theta(\cH-1/2)}(|e^{-i\theta(\cH-1/2)}u|^{2\sigma} e^{-i\theta(\cH-1/2)}u),
\end{split}
\end{equation}
where $F\in C(\R\times \Sigma^2, \Sigma^2)$. Based on (\ref{specn}), the operator $e^{i\theta(\cH-\h)}$ is $2\pi$-periodic with respect to $\theta$; thus, $F$ is also $2\pi$-periodic.
Then, $\phiep$ satisfies the following equation,
\begin{equation}\label{eplmt}
i\partial_t\phiep(t)= \cH_z\phiep(t)+\lmd F(\f{t}{\ep^2},\phiep(t)).
\end{equation}
Defining the averaged function of $F$ by
\begin{equation}\label{Fav}
\Fav(u):=\f{1}{2\pi}\int_0^{2\pi} e^{i\theta\cH}(|e^{-i\theta\cH}u|^{2\sigma}e^{-i\theta\cH}u)d\theta
\end{equation}
 and letting $\ep\to +0$ in (\ref{eplmt}), the limiting model is formally derived as
\begin{equation}
\label{lmt}
i\partial_t \phi = \cH_z \phi+\lmd \Fav(\phi) \tag{L-NLS}
\end{equation}
with the initial data $\phi(0,\bfx)=\psi_0(x,z)$. Here, we considered the fact that $\cH$ commutes with $\partial_z^2$ and $V(z)$.

The nonlinearity of \eqref{Fav} has a similar structure to the resonant system for NLS with harmonic trapping, cf. \cite{LBL, CR, RH}. \\

Note that equation (\ref{lmt}) is a model in 3 spatial dimensions. (\ref{lmt}) describes the resulting averaged particle dynamics. Solution to (\ref{lmt}) conserves mass $M$, defined in (\ref{mass}), and energy $E$,  which is defined as
\begin{equation}\label{en}
E[\phi]:= \h\int_{\R^3} |\partial_z \phi|^2 d\bfx+\int_{\R^3}V(z)|\phi|^2 d\bfx+ \f{\lmd}{2\pi(\sigma+1)} \int_{\R^3}\int_{0}^{2\pi}|e^{-i\theta \cH}\phi |^{2\sigma+2}  d\theta d\bfx.
\end{equation}

\subsection{Previous results}
The local well-posedness of (\ref{ep}) for each $\ep$ in the space $\Sigma^2$ follows immediately from the results in \cite{1}. Then, \cite{2} showed the solutions are uniformly bounded with respect to sufficiently small $\ep$ and converge to the corresponding solution to (\ref{lmt}) in the following sense.
 
 \begin{theorem*}[\cite{2} Theorem $1.1$] 
  Let $V$ satisfy (\ref{V}), $\sigma \in \N$, and $\psi_0 \in \Sigma^2$.
  \begin{enumerate}
\item
There is a $T_{\text{max}} \in (0,\infty]$ and a unique maximal solution $\phi \in C([0,T_{\text{max}}), \Sigma^2) \cap C^1([0,T_{\text{max}}), L^2(\R^3))$ of the limiting equation (\ref{lmt}), such that
\[ \|\phi(t, \cdot)\|_{L^2}=\|\psi_0\|_{L^2}, \quad E[\phi(t, \cdot)]=E[\psi_0], \quad  \forall t\in [0, \Tmax). \]

\item
For all $T \in (0,T_{\text{max}})$ there are $\ep_T >0$, $C_T >0$ such that, for all $\ep\in(0,\ep_T]$, equation (\ref{ep}) admits a unique solution $\psiep \in  C([0,T], \Sigma^2) \cap C^1([0,T], L^2(\R^3))$, which is uniformly bounded with respect to $\ep \in (0, \ep_T]$ in $L^\infty((0,T), \Sigma^2)$ and satisfies 
\[ \max_{t\in [0,T]}\|\psiep(t,\cdot)-e^{-it\cH/\ep^2} \phi(t,\cdot) \|_{L^2} \leqslant C_T \ep^2. \]
\end{enumerate}
 \end{theorem*}

The authors remark that this theorem is in the same spirit as earlier results, such as \cite{N.Ben2, N.Ben, N.Ben3, FF, M}. The averaging method is also used in recent studies, cf. \cite{mana, iwa}.\\
  
 The authors comment in \cite[Remark 2.2]{2} that we would need to treat these solutions in $\Sigma^1=\{ u\in H^1(\R^3): |\bfx|u\in L^2 \}$ to obtain a global (i.e. $\Tmax=+\infty$) result; on the other hand, we need to impose a severe restriction on $\sigma$ to prove Lipschitz continuity of the nonlinearity in 3 dimensions.\\

Our purpose is to obtain global solutions to (L-NLS) (and ($\ep$-NLS)) in $\Sigma^1$ and improve the abovementioned convergence result in time.    
Consequently, we use conservation laws of
\begin{equation}
  \int_{\R^3}  -\f{i}{2}(x^\perp\cdot\nabla_x \psi) \overline{\psi} d\bfx 
\end{equation}
for solution to (\ref{ep}) and 
\begin{equation}\label{ke0}
	K[\phi]:= \h\int_{\R^3} |\nabla_x \phi|^2 d\bfx+ \f{1}{8}\int_{\R^3} |x\phi|^2 d\bfx 
\end{equation}
for solution to (\ref{lmt}). We also use Strichartz's estimates for $\{ e^{it\cH/\ep^2} \}_{t \in \R}$ and $ \{e^{it\cH_z} \}_{t \in \R}$. As a result, in addition to our original goal, we prove a scattering result for the nonic nonlinear power of (\ref{lmt}),  under a scaling-invariant small-energy condition, which indicates energy-criticality of the nonic case.

 \subsection{Main results}
 
In this study, we consider only the cases $\lmd=+1$ or $\lmd=-1$, in which  (\ref{ep}) and (\ref{lmt}) are defocusing or focusing, respectively. We prove the local well-posedness for the equations in the spaces, 
\[\Sigma^1:=\{u \in H^1(\R^3): |\bfx|u \in L^2(\R^3)\} , \quad  \| u\|_{\Sigma^1}:=(\| \nabla_\bfx u \|_{L^2(\R^3)}^2+\| |\bfx|u\|_{L^2(\R^3)}^2)^{1/2}\]
\[L_x^2 \Sigma_z^1:=\{u \in L^2(\R^3): \partial_zu, |z|u \in L^2(\R^3)\},  \quad  \| u\|_{L_x^2 \Sigma_z^1}:=(\| \partial_z u \|_{L^2(\R^3)}^2+\|zu\|_{L^2(\R^3)}^2)^{1/2}\]
and obtain the global solutions mainly in the defocusing case. Based on the global existence of the solution to (\ref{lmt}), we extend the convergence result in \cite{2} over an arbitrarily long time.

Furthermore, we find that (\ref{lmt}) is energy-critical when $\sigma=4$, which is a nonic nonlinear power:
\[ i\partial_t \phi(t) = \cH_z \phi(t)+\f{\lmd}{2\pi}\int_0^{2\pi} e^{i\theta\cH}(|e^{-i\theta\cH}\phi(t)|^{8}e^{-i\theta\cH}\phi(t))d\theta. \]
That is, we prove a scattering result with nonic nonlinear power of (\ref{lmt}),  under a scaling-invariant small-energy condition. This is interesting because (\ref{ep}) is energy-critical in the quintic case. Thus, this property, global well-posedness of (\ref{lmt}), and convergence result lead us expect global existence of the solutions to (\ref{ep}) or NLS with anisotropic harmonic trapping in the energy-supercritical case (e.g. septic case in 3 dimensions) under appropriate conditions. However, in the nonic case, global existence of the solution to (\ref{lmt}) for a large data is an  open problem even in the defocusing case. \\

We will now state the main results. In all following theorems, we weaken the assumption \eqref{V} and assume that there exists $V''$, which is 
\begin{equation} \label{asV}
	V''\in L^\infty(\R).
\end{equation}

\begin{thm}[Well-posedness of the magnetic Schr\"{o}dinger equation]\label{main1} \quad\\
For $\sigma =1,2$ and each $\ep>0$, (\ref{ep}) is locally well-posed in $\Sigma^1$.  Suppose that  $[0, \Tmax^\ep)$ is the maximal lifespan of such solution.  
\begin{enumerate}
\item In the case $\sigma=1$.
\begin{itemize}
\item 
If $\lmd =+1$, $\Tmax^\ep=+\infty$.
\item
If $\lmd=-1$ and $V$ is bounded below, there exists $\ep_*=\ep_*(\|\psi_0\|_{\Sigma^1})>0$ such that for all $\ep \in (0,\ep_*]$, $\Tmax^\ep=+\infty$.  
\item
In the other cases, (that is all cases,) $\Tmax^\ep \to +\infty$ as $\ep\to+0$.	
\end{itemize}

\item In the case $\sigma=2$, if $\lmd=+1$, it holds $\Tmax^\ep \to +\infty$ as $\ep\to+0$.
\end{enumerate}

\end{thm}

\begin{thm}[Well-posedness of the limit equation]\label{main2} 
\quad\\
\begin{enumerate}
\item For $\sigma=1$, (\ref{lmt}) is globally well-posed in $L_x^2\Sigma_z^1$. 

\item For $\sigma=2,3,4$, (\ref{lmt}) is locally well-posed in $\Sigma^1$. If $\sigma =2,3$  and $\lmd=+1$, the solution exists globally.
\end{enumerate}
\end{thm}

\begin{remark}
In each case, persistence of regularity also holds. That is, if an initial data belongs to $\Sigma^k$ for some integer $k$, the corresponding solution to (\ref{ep}) or (\ref{lmt}) maintains the same regularity as long as the solution exists in the larger space.\\
\end{remark}

Next, we consider the case $\sigma=4$ and $V\equiv0$ in the space 
\[ \Sigma_{0}^1=\{u \in H^1(\R^3):  |x| u \in L^2(\R^3)\},  \quad  \| u\|_{\Sigma_{0}^1}:=(\| \nabla_\bfx u \|_{L^2}^2+\||x|u\|_{L^2}^2)^{1/2}\]
instead of $\Sigma^1$. In the following theorem, we also use the norm
\[ \|u \|_{L_z^2\Sigma_x^1}:=(\|\nabla_x\|_{L_\bfx^2}^2+ \||x|u\|_{L_\bfx^2}^2)^\h .\] 

\begin{thm}[Scale-invariance and small data scattering of the limit equation]\label{scatter}
\quad Let $\sigma=4$ and $V\equiv0$.
\begin{enumerate}
\item
 (\ref{lmt}) is left-invariant by the scaling
\begin{equation}\label{scale}
\phi(t,x,z) \longmapsto \phi_\mu(t,x,z):=\mu^{\f{1}{4}}\phi(\mu^2t,x,\mu z) \quad \mu>0.
\end{equation} 
This scaling leaves neither $ K[\phi_\mu(t)]$ nor $E[\phi_\mu(t)]$ invariant, but conserves
\[ K[\phi_\mu(t)]^3E[\phi_\mu(t)]. \]
Here, $K$ and $E$ are defined in \eqref{ke0} and \eqref{en}, respectively.
\item
There exists $\delta>0$ such that for any $\psi_0\in \Sigma_0^1 $ satisfying the scale invariant condition $\|\psi_0\|_{L_z^2\Sigma_x^1}^3\|\partial_z \psi_0\|_{L^2} \le \delta$, there exists a unique global solution to  (\ref{lmt}) $\phi  \in C(\R, \Sigma_0^1)$ with the data $\psi_0$. Moreover, there exist functions $\phi_{\pm}\in \Sigma_0^1$ such that 
\begin{equation}
 \lim_{t\to\pm\infty}\|\phi(t) - e^{\pm it \partial_z^2/2}\phi_{\pm}\|_{\Sigma_0^1}=0 .
\end{equation}
\end{enumerate}
\end{thm}

Next, we present the convergence of (\ref{smc}). Because we obtain the global solution to (\ref{lmt}) under some conditions, in these cases  the following convergence result holds over any compact time interval.
\begin{thm}[Strong magnetic confinement limit]\label{main3} \qquad
\begin{enumerate}
\item
Let $\sigma\in \{1,2\}$ and $\phi \in  C([0,\Tmax), \Sigma^1)$ be a maximal solution to (\ref{lmt})  with a data $\psi_0\in \Sigma^1$. Then, for any $T\in (0,\Tmax)$ there exists a solution to (\ref{ep}) $\psiep \in  C([0, T], \Sigma^1)$ with the data $\psi_0$ for all sufficiently small $\ep>0$, which satisfies
\begin{equation}\label{main3-1}
\lim_{\ep \to +0}\| e^{it\cH/\ep^2}\psiep-\phi\|_{L^\infty ([0,T],\Sigma^1)}=0.
\end{equation}
Furthermore, there exist $ \ep_0=\ep_0(\|\psi_0\|_{\Sigma^1},T)>0$ and $C=C(\|\psi_0\|_{\Sigma^1},T)>0$ such that for any $\ep \in (0 , \ep_0]$,
\begin{equation}\label{main3-2}
\|e^{it\cH/\ep^2}\psiep-\phi\|_{L^\infty([0,T],L^2)}\le C\ep .
\end{equation}

\item
Let $\sigma\in \N$ and  $\phi \in  C([0,\Tmax), \Sigma^2)$  be a maximal solution to (\ref{lmt}) with a data $\psi_0\in \Sigma^2$. Then, for any $T \in (0,\Tmax)$ there exists a solution to (\ref{ep}) $\psiep \in  C([0, T], \Sigma^2)$ with the data $\psi_0$ for all sufficiently small $\ep>0$, which satisfies
\begin{equation}\label{main3-3}
\lim_{\ep \to +0}\|e^{it\cH/\ep^2}\psiep-\phi\|_{L^\infty([0,T], \Sigma^2)}=0 .
\end{equation}
\end{enumerate}
\end{thm}

On the right-hand side of (\ref{main3-2}), the convergence rate $\ep$ is due to the difference of the first-order regularity between the spaces $\Sigma^1$ and $L^2$. We do not know whether this rate is optimal.\\

\textbf{Organization of the paper.} Section \ref{Pre} introduces the notation and basic estimates used in this study. Section \ref{Wpep} presents Theorem \ref{main1}. Section \ref{Wplmt} presents Theorem \ref{main2} and Theorem \ref{scatter}. Section \ref{Sml} presents Theorem \ref{main3}.  Section \ref{Ap} introduces some properties of Harmonic oscillator required in this study.

\section{Preliminaries}\label{Pre}

\subsection{Notation}

We write $X\lesssim Y$ to express $X\le CY$ for some constant $C$.  $L^p(\R^d)$ denotes the usual Lebesgue spaces. We often  use the following notation
\[\|u\|_{L_\bfx^p} =\Big( \int_{\R^3} |u(\bfx)|^p d\bfx \Big)^\f{1}{p}\]
and  the partial spatial norm (recall $\bfx=(x,z)=(x_1,x_2,z)$)
\[ \|u\|_{L_x^p} =\Big( \int_{\R^2} |u(\bfx)|^p dx \Big)^\f{1}{p} \]
\[ \|u\|_{L_z^p} =\Big( \int_{\R} |u(\bfx)|^p dz \Big)^\f{1}{p} .\]
If $I\subset \R$  is an interval, the mixed Lebesgue norm on $I\times\R^3$ are defined by
\[ \|u\|_{L_t^qL^p_\bfx(I\times\R^3)}= \|u\|_{L^q(I,L^p(\R^3))}=\Big(  \int_{I} (\int_{\R^3} |u(t,\bfx)|^p d\bfx )^{\f{q}{p}}\Big)^{\f{1}{q}}. \]
Similarly, we denote
\[ \|u\|_{L_x^{p_1}L_z^{p_2}(\R^3)} = \Big(  \int_{\R^2} (\int_{\R} |u(\bfx)|^{p_2} dz )^{\f{p_1}{p_2}} dx \Big)^{\f{1}{p_1}} \]
\[ \|u\|_{L_z^{p_2}L_x^{p_1}(\R^3)} = \Big(  \int_{\R} (\int_{\R^2} |u(\bfx)|^{p_1} dx )^{\f{p_2}{p_1}} dz \Big)^{\f{1}{p_2}}. \]
When $X$, $Y$ are different norms, we denote
\[ \|u\|_{X\cap Y}:=\|u\|_X +\|u\|_Y.\]
If there is no confusion, we often omit the integral regions. \\

We use Hermite Sobolev space
\[ \Sigma^k=\Sigma_{\bfx}^k:=\{u\in H^k(\R^3):|\bfx|^k u\in L^2(\R^3)\}\]
\begin{equation*}
\begin{split}
 \|u\|_{\Sigma^k}&:= \Big( \sum_{j=1}^3 \| \partial_{x_j}^k u \|_{L^2(\R^3)}^2 +\| |\bfx|^ku \|_{L^2(\R^3)}^2 \Big)^\h  \qquad (x_3=z)
\end{split}
\end{equation*}
and denote partial spatial norm $  \|u\|_{\Sigma_x^1} $ by
\[ \|u\|_{\Sigma_x^k}= \Big( \int_{\R^2} |\partial_{x_1}^k u|^2 +  |\partial_{x_2}^k u|^2 + |x|^{2k}|u|^2 dx \Big)^\h . \]
We also define $\|u\|_{\Sigma_z^1}$ similarly. Note
\begin{equation*}
\begin{split}
\|u\|_{\Sigma^k}&\simeq_k \Big( \int_{\R^3}\sum_{|\al|\le k} | D_\bfx^\al u|^2 +  |\bfx|^{2k}|u|^2 d\bfx \Big)^\h \\
\end{split}
\end{equation*}
where 
\[ \mathbf{\al} =(\al_1, \al_2,\al_3) \in \Z_{\ge0}^3\quad  , \qquad D_\bfx^\al=\partial_{x_1}^{\al_1}\partial_{x_2}^{\al_2}\partial_{x_3}^{\al_3}=\partial_{x_1}^{\al_1}\partial_{x_2}^{\al_2}\partial_{z}^{\al_3} .\]
To obtain a solution to (\ref{ep}) and (\ref{lmt}),  we define the function space $\Sigma_{x,z}^{p_1,p_2,k}$ for $1\le p_1,p_2\le \infty$ and $k\in \N$ as follows,
\[ \Sigma_{x,z}^{p_1,p_2,k} := \{u\in L_x^{p_1}(\R^2,L_z^{p_2}(\R)) : \|u\|_{\Sigma_{x,z}^{p_1,p_2,k}} < \infty \} \]
\[ \| u \|_{\Sigma_{x,z}^{p_1,p_2,k}} : = \sum_{ |\al|\le k} \| D_\bfx^\al u\|_{L_x^{p_1}L_z^{p_2}(\R^3)} + \| \langle \bfx \rangle^k u \|_{L_x^{p_1} L_z^{p_2}(\R^3)}, \]
where $\langle \bfx \rangle =(1+|\bfx|^2)^{1/2}$. 
By switching the order of the norm, define $\Sigma_{z,x}^{p_2,p_1,k}$ in the same manner.  When $k=1$, we omit $k$.\\

We decompose $\cH$ as
\[ \cH= -\f{1}{2}\Delta_x+\f{1}{8}|x|^2-\f{i}{2}x^\perp\cdot\nabla_x := \cH_0+L\]
\[ \text{where \quad} \cH_0 := -\f{1}{2}\Delta_x+\f{1}{8}|x|^2  \quad \text{and}\quad  L:= -\f{i}{2}x^\perp\cdot\nabla_x.\]
Note that $\cH$ and $\cH_0$ are commutative. From ~\cite[Lemma 2.4]{Yajima},  for $1<p_1<\infty$ and $s \ge 0$, one has the norm equivalence
\begin{equation}\label{eqH}
\| \mathcal{F}^{-1} \langle \xi \rangle^s \mathcal{F} u\|_{L_x^{p_1} } +\| \langle x \rangle^s u \|_{L_x^{p_1}} \simeq \| \cH_0^{s/2} u \|_{L_x^{p_1}}, 
\end{equation}
where $\mathcal{F}$ is the Fourier transform in $x$.  In fact, this holds for all dimension.

\subsection{Basic estimates}

\begin{lemma}\label{sigmaz} For any $k\in \N$ and $f\in \Sigma^k$,
\[  \|e^{-it\cH_z} f  \|_{\Sigma_z^k}  \le e^{C_k|t|} \| f \|_{\Sigma_z^k}.  \]
\end{lemma}

\begin{lemma}\label{sigmaz2} For any $k\in \N$ and $f\in \Sigma^{k+2}$,
\[  \| \cH_z f  \|_{\Sigma_z^k}  \lesssim_k  \| f \|_{\Sigma_z^{k+2}}.  \]
\end{lemma}
\noindent
Proof of these lemmas are  the same  as  \cite[Lemma 3.2]{2}.
\begin{lemma}[\cite{StH} Strichartz's estimate for $\{e^{it\cH} \}_{t\in\R}$]
\quad\\
Let $(p_0,q_0)$ and $(p_1,q_1)$ be 2-dimensional admissible pairs,  that is for $j=0,1$
\[ 2\le p_j < \infty   \qquad  \text{and} \qquad \f{2}{q_j}+\f{2}{p_j}=1.\]
Then for any  $T>0$, 
\[ \Big\| e^{it\cH}f \Big\|_{L^{q_0}([0,T], L^{p_0})} \lesssim (1+T)^{1/q_0} \| f \|_{L^2} \]
\[ \Big\| \int_0^t e^{i(t-s)\cH} g(s, \cdot) ds \Big\|_{L^{q_0}([0,T], L^{p_0})}   \lesssim (1+T)^{1/q_0 +1/q_1} \| g \|_{L^{q_1'}([0,T],L^{p_1'})}.\]
\end{lemma}

\quad\\

Based on 
\[  e^{it\cH} = e^{it\cH_0} e^{itL}= e^{itL} e^{it\cH_0} \]
(see Section \ref{Ap}) and  $e^{itL}$ being the rotation of angle $t$ around the origin (cf.~\cite{CR}), proof of this lemma comes down to  Strichartz's estimate for $\{ e^{it\cH_0} \}_{t\in\R}$.

In this study, by changing the variable of time, we use Strichartz's estimate in the following form, 
\begin{equation}\label{Step1}
\Big\| e^{it\cH/\ep^2} f \Big\|_{L^{q_0}([0,T], L^{p_0})} \lesssim (\ep^2 +T)^{1/q_0} \| f \|_{L^2} 
\end{equation}
\begin{equation}\label{Step2}
\Big\| \int_0^t e^{i(t-s)\cH/\ep^2 } g(s, \cdot) ds \Big\|_{L^{q_0}([0,T], L^{p_0})}   \lesssim (\ep^2 +T)^{1/q_0+1/q_1}\| g \|_{L^{q_1'}([0,T], L^{p_1'})} .
\end{equation}
\quad\\

\noindent
We also use Strichartz's estimate for $\{e^{it\cH_z}\}_{t\in\R}$.
\begin{prop}[\cite{Fu79,Fu80}Dispersive estimate for  $\{e^{it\cH_z}\}_{t\in\R}$]\quad\\
Let $V$ be smooth and subquadratic. Then, there exists $\delta>0$ such that  for all $t \in (-\delta, \delta)$ one has
\[ \| e^{it\cH_z}u \|_{L_z^\infty}\lesssim |t|^{-\h}\|u\|_{L_z^1}. \]
\end{prop}

\noindent
From this proposition, we have the following estimates.
\begin{lemma}[Strichartz's estimate for $\{e^{it\cH_z} \}_{t\in\R}$] \quad\\
Let $(p_0,q_0)$ and $(p_1,q_1)$ be 1-dimensional admissible pairs,  that is for $j=0,1$
\[ 2\le p_j \le \infty   \qquad  \text{and} \qquad \f{2}{q_j}+\f{1}{p_j}=\h.\]
Then, for any  $T>0$, 
\[ \Big\| e^{it\cH_z}f \Big\|_{L^{q_0}([0,T], L^{p_0})} \lesssim \Big(1+\f{T}{\delta}\Big)^{1/q_0} \| f \|_{L^2} \]
\[ \Big\| \int_0^t e^{i(t-s)\cH_z} g(s, \cdot) ds \Big\|_{L^{q_0}([0,T], L^{p_0})}   \lesssim \Big(1+\f{T}{\delta} \Big)^{1/q_0 +1/q_1} \| g \|_{L^{q_1'}([0,T],L^{p_1'})},\]
where $\delta$ is the constant appearing in the above proposition.
\end{lemma}
\quad\\
\noindent
In this paper, the dependence of constants on $\delta$ is omitted.

\subsection{Conservation}

Next, we derive two conservation laws. 

\begin{lemma}
	Solution to (\ref{ep}) conserves angular momentum in $x$ 
\begin{equation}\label{ang}
 \langle L\psi,\psi \rangle_{L_\bfx^2}=\int_{\R^3} (L\psi)\overline{\psi} d\bfx = \int_{\R^3}  -\f{i}{2}(x^\perp\cdot\nabla_x \psi) \overline{\psi} d\bfx .
\end{equation}

\end{lemma}

\begin{proof}
	Let $T>0$ and $\psi(t) \in C([0,T], \Sigma^2)$ be a solution to (\ref{ep}) for some $\ep$. Then, 
	\begin{equation}
		\begin{split}
			\f{d}{dt} \langle L \psi(t), \psi(t) \rangle 
			&= 2 \lmd\text{Re} \langle L \psi(t), (-i)|\psi(t)|^{2\sigma}\psi(t) \rangle \\
			&=2\lmd \text{Re} \int_{\R^3} iL\psi(t) |\psi(t)|^{2\sigma} \overline{\psi(t)}  d\bfx,
		\end{split}
	\end{equation}
where we use the fact $\cH$ is commutative with $\cH_0$ and therefore also with $L$.
	Let $\tilde{L}:=iL=(-x_2\partial_{x_1}+x_1\partial_{x_2})/2$. Then, we have for any function $\psi \in \Sigma^2$ that
	\begin{equation}\label{Im}
	\begin{split}
		&2\text{Re}\int_{\R^3} (iL\psi)|\psi|^{2\sigma}\overline{\psi}
		=\int_{\R^3} (\tilde{L}|\psi|^2)|\psi|^{2\sigma}
		=\f{1}{\sigma+1}\int_{\R^3}\tilde{L}(|\psi|^{2\sigma+2}
)		=0.
	\end{split}
	\end{equation}

\end{proof}

\begin{remark}
	By Theorem \ref{main1}, when $\sigma=1,2$ we have the well-posedness of \eqref{ep} in $\Sigma^1$. Because $\Sigma^2$ is dense in $\Sigma^1$, we can extend the conservation law of \eqref{ang} to the solutions that belong to $C([0,T],\Sigma^1)$.

	Fix $\ep$ and let $\psi \in C([0,T], \Sigma^1)$ be the solution to \eqref{ep}, where $T$ is some positive constant. By the density, for any small $\delta>0$, there exists $\psi_{0, \delta} \in \Sigma^2$ such that 
		\begin{equation*}
		\|\psi(0)-\psi_{0,\delta}\|_{\Sigma^1}\le \delta.
	\end{equation*}
	From the local well posedness of \eqref{ep} in $\Sigma^1$ and persistence of regularity, there exists $\delta_0>0$ depending on $T$ such that,  for any $\delta\in (0,\delta_0)$ there exists the solution to \eqref{ep} $\psi_\delta \in C([0,T], \Sigma^2)$ with $\psi_\delta(0)=\psi_{0,\delta}$ which satisfies
	\begin{equation*}
		\|\psi-\psi_\delta\|_{L^\infty([0,T],\Sigma^1)} \le C(\psi(0), T)\|\psi(0)-\psi_{0,\delta}\|_{\Sigma^1}.
	\end{equation*}
	Because $\langle L\psi_\delta(t),\psi_\delta(t) \rangle_{L_\bfx^2}=\langle L\psi_{0,\delta},\psi_{0,\delta} \rangle_{L_\bfx^2}$ holds for any $t\in [0,T]$, by the triangle inequality and  H\"{o}lder's inequality,  we have
	\begin{equation*}
	\begin{split}
		&|\langle L\psi(t),\psi(t) \rangle_{L_\bfx^2}-\langle L\psi(0),\psi(0) \rangle_{L_\bfx^2}| \\
		\le& |\langle L\psi(t),\psi(t) \rangle_{L_\bfx^2}-\langle L\psi_\delta(t),\psi_\delta(t) \rangle_{L_\bfx^2}|+|\langle L\psi_{0,\delta},\psi_{0,\delta} \rangle_{L_\bfx^2}-\langle L\psi(0),\psi(0) \rangle_{L_\bfx^2}|\\
		\le& C(\psi(0) ,T)  \delta.
	\end{split}
	\end{equation*}
	Letting $\delta\to0$, we have  $\langle L\psi(t),\psi(t) \rangle_{L_\bfx^2}=\langle L\psi(0),\psi(0) \rangle_{L_\bfx^2}$ for any $t\in [0,T]$. Iterating this argument,   $\langle L\psi(t),\psi(t) \rangle_{L_\bfx^2}$ conserves as long as $\psi(t)$ exists. 
\end{remark}

\quad

Then, we define a conservative quantity 
\begin{equation}\label{eep0}
\begin{split}
E_0^\ep[\psi]&:= \h E^\ep[\psi]-\f{1}{\ep^2}\int_{\R^3} (L\psi)\overline{\psi} d\bfx \\
&=\f{1}{\ep^2}\Big(\h\| \nabla_x \psi\|_{L^2}^2+\f{1}{8} \| |x|\psi \|_{L^2}^2\Big)+ \f{1}{2}\| \partial_z \psi\|_{L^2}^2 + \int_{\R^3} V(z) |\psi|^2 d\bfx  \\
 & \qquad + \f{\lmd}{\sigma+1}\| \psi \|_{L^{2\sigma+2}}^{2\sigma+2}.
\end{split}
\end{equation}
$E^\ep[\psi]$ cannot control $\|\psi\|_{L_z^2\Sigma_x^1}$(See Section \ref{Ap}). However, if $\lmd=+1$ and $\sigma=1,2$, 
\[ \|\psi\|_{\Sigma^1}^2 \lesssim E_0^\ep[\psi] +c\|\langle z \rangle \psi\|_{L^2}^2\]
 holds (See Section \ref{Wpep}). Therefore, in this study we use $E_0^\ep$  as ``energy''. \\

On the other hand, (\ref{lmt}) consists of, in addition to mass $M$ and Hamiltonian $E$, the following conservation law. 
\begin{lemma}
Solution to (\ref{lmt})  conserves 
\begin{equation}\label{ke}
K[\phi]:=\langle \cH_0\phi, \phi\rangle = \h \| \nabla_x \phi \|_{L^2}^2+\f{1}{8}\| x \phi\|_{L^2}^2 .
\end{equation}
\end{lemma}
\noindent
Proof of this lemma is in the same spirit as \cite[Section 3]{MS}. 
\begin{proof}
We prove the case $\sigma=1$. Other cases can be proved in the same manner. Recall that $\cH_0$ is commutative with $\cH$ and therefore also with $e^{i\theta\cH}$. Then, let $\phi \in C([0,T], \Sigma^2)$ be the solution to (\ref{lmt}), we have for any $t \in[0,T]$ that
\begin{equation*}
\begin{split}
\f{d}{dt} \langle \cH_0 \phi(t), \phi(t) \rangle &= -2 \lmd\text{Im} \langle \cH_0\phi(t), \Fav{\phi(t)} \rangle \\
&=2\lmd \text{Im} \f{1}{2\pi} \int_0^{2\pi} \int_{\R^3} |e^{-i\theta\cH}\phi(t)|^{2} e^{-i\theta\cH}\phi(t) \overline{\cH_0 e^{-i\theta\cH}\phi(t)} d\bfx d\theta \\
&=2\lmd \text{Im} \f{1}{2\pi} \int_0^{2\pi} \int_{\R^3} |e^{-i\theta\cH}\phi(t)|^{2} e^{-i\theta\cH}\phi(t) \overline{\cH e^{-i\theta\cH}\phi(t)} d\bfx d\theta.
\end{split}
\end{equation*}
In the last line, we used the fact for any function $u \in \Sigma^2$,
\[  \text{Im}  \int_{\R^2} |u|^{2} u \overline{Lu} dx =0 .\]
On the other hand, in Section \ref{Ap}, we denote the Hermite expansion for $\cH$ of $\phi(t)$ as
\[ \phi(t,\bfx)= \sum_{n=0}^\infty c_n(t) \phi_n(t, \bfx). \]
Then,
\begin{equation*}
\begin{split}
&\text{Im} \f{1}{2\pi} \int_0^{2\pi} \int_{\R^3}  |e^{-i\theta\cH}\phi(t)|^{2} e^{-i\theta\cH}\phi(t) \overline{\cH e^{-i\theta\cH}\phi(t)} d\bfx d\theta \\
 &= \text{Im}\sum_{n_1+n_2 =  n_3+n_4}c_{n_1}c_{n_2}  \overline{ c_{n_3}c_{n_4} }(n_4+\h)   \phi_{n_1} \phi_{n_2} \overline{\phi_{n_3} \phi_{n_4}}\\
 &= -\f{1}{4} \text{Im}\sum_{n_1+n_2= n_3+n_4}
  (n_1+n_2-n_3-n_4) c_{n_1}c_{n_2}  \overline{ c_{n_3}c_{n_4} } \phi_{n_1} \phi_{n_2}  \overline{  \phi_{n_3}\phi_{n_4} }=0.
\end{split}
\end{equation*}
\end{proof}

\begin{remark}
	By Theorem \ref{main2}, when $\sigma=1,2,3,4$ we have the well-posedness of \eqref{lmt} in $\Sigma^1$. Because $\Sigma^2$ is dense in $\Sigma^1$, we can extend the conservation law of $K$ to the solutions that belong to $C([0,T],\Sigma^1)$. 
\end{remark}

\section{Well-posedness of (\ref{ep})}\label{Wpep}

\subsection{ The case $\sigma=1$}

We now present the precise statement.

\begin{thm}\label{wpep1}
Let $\sigma=1$ and $V$ satisfy \eqref{asV}. For any data $\psi_0 \in \Sigma^1$ and any $\ep>0$, there exist $T>0$ and a unique solution to (\ref{ep})  $\psi^\ep \in C([0, T], \Sigma^1)\cap L^4([0, T],\Sigma_{x,z}^{4,2})$, depending continuously on $\psi_0$. Suppose $\Tmax^\ep \in (0,+\infty]$ is the maximal time of existence. 
\begin{itemize}
\item
If $\lmd =+1$, $\Tmax^\ep=+\infty$.
\item
If $\lmd=-1$ and $V$ is bounded below, there exists $\ep_*=\ep_*(\|\psi_0\|_{\Sigma^1})>0$ such that for all $\ep \in (0,\ep_*]$, $\Tmax^\ep=+\infty$.  
\item
In the other cases, $\Tmax^\ep \to +\infty$ as $\ep\to+0$.
\end{itemize}
\end{thm}

\begin{proof}
(\ref{ep}) is equivalent to the integral equation
\begin{equation}\label{ep2}
\psiep(t) =  e^{-it(\cH/\ep^2+\cH_z)} \psi_0 -i\lmd \int_0^t e^{-i(t-s)(\cH/\ep^2+\cH_z)}[|\psiep|^2\psiep](s) ds .
\end{equation}
For $a>0$ and $0<T\le1$, we define
\[ M(a,T)=\{ \psi \in L^\infty([0,T],\Sigma^1) \cap L^4([0,T],\Sigma_{x,z}^{4,2}) : \| \psi \|_{L_t^\infty \Sigma^1\cap L_t^4\Sigma_{x,z}^{4,2}([0,T]) } \le a \} \]
and 
\[ \Psi^\ep[\psi] :=  e^{-it(\cH/\ep^2+\cH_z)} \psi_0 -i\lmd \int_0^t e^{-i(t-s)(\cH/\ep^2+\cH_z)}[|\psi|^2\psi](s) ds .\]
We will choose $T$ and $a$ so that $\Psi^\ep :M(a,T) \to M(a,T)$ and is a contraction.
By (\ref{eqH}), the unitarity of $e^{it\cH_z}$ in $L_z^2(\R)$,  Minkowski's inequality, Strichartz's estimate (\ref{Step1}), and (\ref{Step2}),
\begin{equation*}
\begin{split}
\| \langle x \rangle \Psi^\ep[\psi]& \|_{L_t^{\infty} L_\bfx^2([0,T])} +\| \nabla_x  \Psi^\ep[\psi] \|_{L_t^{\infty} L_\bfx^2([0,T])}   \\
&\lesssim \| \psi_0 \|_{L_z^2 \Sigma_x^1}  +  \Big\| \int_0^t e^{-i(t-s)(\cH/\ep^2)}\cH_0^\h e^{is\cH_z}[|\psi|^2\psi](s) ds \Big\|_{L_t^\infty L_x^2L_z^2} \\
 &\lesssim \| \psi_0 \|_{L_z^2 \Sigma_x^1} + \| \cH_0^\h e^{it\cH_z} |\psi|^2\psi  \|_{L_z^2 L_{t,x}^{4/3} } .\\
\end{split}
\end{equation*}
Using (\ref{eqH}) and Minkowski again, we have
\begin{equation*}
\begin{split}  \| \cH_0^\h e^{it\cH_z} |\psi|^2\psi  \|_{ L_z^2 L_{t,x}^{4/3} } \lesssim&  \| \langle x \rangle e^{it\cH_z} |\psi|^2\psi  \|_{ L_z^2 L_{t,x}^{4/3}}   +  \|  \nabla_x e^{it\cH_z} |\psi|^2\psi  \|_{ L_z^2 L_{t,x}^{4/3}} \\
\lesssim&  \| \langle x \rangle  |\psi|^2\psi  \|_{L_{t,x}^{4/3} L_z^2}   +  \|  \nabla_x  |\psi|^2\psi  \|_{L_{t,x}^{4/3} L_z^2}.
\end{split}
\end{equation*}
By H\"{o}lder's inequality and Gagliardo-Nirenberg's inequality for $L_z^\infty$,
\begin{equation}\label{Psinl}
\begin{split}
&\| \langle x\rangle   |\psi|^2\psi  \|_{L_{t,x}^{4/3} L_z^2}   +  \| \nabla_x |\psi|^2\psi \|_{L_{t,x}^{4/3} L_z^2}\\
\lesssim& \Bigl\| \| \langle x \rangle \psi\|_{L_z^2} \| \psi\|_{L_z^\infty}^2  \Bigr\|_{L_{t,x}^{4/3}}   +  \Bigl\| \| \nabla_x \psi \|_{L_z^2} \| \psi\|_{L_z^\infty}^2  \Bigr\|_{L_{t,x}^{4/3}} \\
\lesssim&  (\| \langle x \rangle \psi \|_{L_{t,x}^4L_z^2}+\| \nabla_x \psi \|_{L_{t,x}^4L_z^2})\| \psi \|_{L_{t,x}^4L_z^2} \| \partial_z\psi \|_{L_{t,x}^4L_z^2} 
\end{split}
\end{equation}
By Minkowski's inequality, Sobolev's embedding, and H\"{o}lder's inequality,  
\begin{equation*}
\| \psi \|_{L_{t,x}^4 L_z^2([0,T])} \lesssim  \| \psi \|_{L_t^4L_z^2 L_x^4}  \lesssim   \| \psi \|_{L_t^4L_z^2 H_{x}^1} \lesssim  \| \psi \|_{L_t^\infty \Sigma^1([0,T])}T^{\f{1}{4}}.
\end{equation*}
Therefore,
\begin{equation*}
\begin{split}
\| \langle x \rangle \Psi^\ep[\psi] \|_{L_t^\infty L_{\bfx}^2}& +\| \nabla_x \Psi^\ep[\psi] \|_{L_t^\infty L_{\bfx}^2}   \\
&\lesssim \| \psi_0 \|_{L_z^2 \Sigma_x^1} + (\| \langle x \rangle \psi \|_{L_{t,x}^4L_z^2}+\| \nabla_x \psi \|_{L_{t,x}^4L_z^2})\| \psi \|_{L_t^\infty \Sigma^1}^2T^{\f{1}{4}}. \\
\end{split}
\end{equation*}
If we change the $L_t^\infty L_{\bfx}^2$ norm to  the $L_{t,x}^4L_z^2$ on the left-hand side, we obtain the same bound.  For derivatives and weights in the $z$-direction, we use Lemma \ref{sigmaz} to commute $\partial_z$, $z$ and $e^{it\cH_z}$. Then, we also have
\begin{equation*}
\begin{split}
\| \langle z \rangle \Psi^\ep[\psi] &\|_{L_t^\infty L_{\bfx}^2 \cap L_{t,x}^4 L_z^2([0,T])} +\| \partial_z \Psi^\ep[\psi] \|_{L_t^\infty L_{\bfx}^2 \cap L_{t,x}^4L_z^2([0,T])}   \\
&\lesssim \| \psi_0 \|_{L_x^2 \Sigma_z^1} + (\| \langle z \rangle \psi \|_{L_{t,x}^4L_z^2}+\| \partial_z \psi \|_{L_{t,x}^4L_z^2})\| \psi \|_{L_t^\infty \Sigma^1}^2T^{\f{1}{4}}. \\
\end{split}
\end{equation*}
Note that because we assume $T\le1$, Lemma \ref{sigmaz} implies 
\[ \|e^{-it\cH_z}u\|_{\Sigma_z^1} \le C \|u\|_{\Sigma_z^1} \qquad  t\in [0,T]\]
for an absolute constant $C$.
Therefore, we obtain
\begin{equation} \label{Psi}
\begin{split} 
 \| \Psi^\ep[ \psi] \|_{L_t^\infty \Sigma^1\cap L_t^4\Sigma_{x,z}^{4,2}([0,T])} &\le C_0\| \psi_0 \|_{\Sigma^1} + CT^{\f{1}{4}} \| \psi\|_{L_t^\infty \Sigma^1 \cap L_t^4\Sigma_{x,z}^{4,2}(0,T)}^3 \\
 &\le  C_0\| \psi_0 \|_{\Sigma^1}+CT^{\f{1}{4}}a^3. 
\end{split}
\end{equation}
for some $C_0\ge1$. If we choose
\begin{equation} \label{T1} 
a=2C_0\| \psi_0 \|_{\Sigma^1} \qquad T\le  \min \{1, (\f{1}{2C a^2})^4\}  
\end{equation}
$\Psi^\ep$ is a mapping on $M(a,T)$. \\

Next, for the contraction, we use the same argument. Then, we have
\begin{equation*}
\begin{split}
 &\| \Psi^\ep[\psi_1]-\Psi^\ep[\psi_2] \|_{L_t^\infty \Sigma^1\cap L_t^4\Sigma_{x,z}^{4,2}([0,T])} \\
 &\lesssim \|\psi_1-\psi_2\|_{L_t^\infty \Sigma^1 \cap L_t^4\Sigma_{x,z}^{4,2}} (\|\psi_1\|_{L_t^\infty \Sigma^1 \cap L_t^4\Sigma_{x,z}^{4,2}}^2+\|\psi_2\|_{L_t^\infty \Sigma^1\cap L_t^4\Sigma_{x,z}^{4,2}}^2)T^{\f{1}{4}}\\
&\lesssim \| \psi_1-\psi_2 \|_{L_t^\infty \Sigma^1 \cap L_t^4\Sigma_{x,z}^{4,2}} T^{\f{1}{4}}a^2.
\end{split}
\end{equation*}
Thus, we establish the contraction property. Uniqueness and continuity statements are easy consequences of the fixed point argument. Persistence of regularity is obtained based on the following estimate. For any $k\in \N$, by (\ref{eqH}), Lemma \ref{sigmaz}, and interpolation 
\[ \|\partial_z^{l} u\|_{L_z^2}\lesssim \|\partial_z^k u \|_{L_z^2}^{\f{l}{k}} \|u\|_{L_z^2}^{\f{k-l}{k}} \qquad l=1,2, \cdots, k-1\]
we have
\begin{equation*}
\|\psi\|_{L_t^\infty \Sigma^k \cap L_t^4 \Sigma_{x,z}^{4,2,k}([0,T])} \le C_k\|\psi_0\|_{\Sigma^k}+  C_k T^{\f{1}{4}} \| \psi\|_{L_t^4\Sigma_{x,z}^{4,2,k}([0,T])}\| \psi \|_{L_t^\infty \Sigma^1([0,T])}^2.
\end{equation*}
Hence, the solution in $L_t^\infty \Sigma^k \cap L_t^4 \Sigma_{x,z}^{4,2,k}$ maintains the same regularity as long as it exists in a larger space  $L_t^\infty \Sigma^1 \cap L_t^4 \Sigma_{x,z}^{4,2}$.\\

To prove the properties of $\Tmax^\ep$, we estimate $\| \psi(t) \|_{\Sigma^1}$.\\

\noindent
\underline{Case 1: $\lmd=+1$} \\

Because $V(z)$ is sub-quadratic, there exists $c\ge1$ such that
\[ 0 \le V(z)+c(1+z^2) \simeq \langle z \rangle^2 . \]
Then, 
\begin{equation}
\label{eqSigma}
 \| \psi(t) \|_{\Sigma^1}^2 \simeq  K[\psi(t)] + B[\psi(t)] +c(M[\psi_0] +\| z \psi(t) \|_{L^2}^2) 
\end{equation}
where $K$ is the quantity defined in \eqref{ke} (which does not conserve in \eqref{ep}) and  
\[ B[\psi]:= \langle \cH_z\psi, \psi \rangle =\h \| \partial_z \psi \|_{L^2}^2 + \int_{\R^3} V(z)| \psi|^2 d\bfx. \]
Then, it holds
\begin{equation*}
\begin{split}
\| \psi(t) \|_{\Sigma^1}^2 &\lesssim E_0^\ep[\psi_0]  +cM[\psi_0]+ c\| z \psi(t)\|_{L^2}^2\\
&\le \big(|E_0^\ep[\psi_0]+cM[\psi_0]|^\h+ c^\f{1}{2}\| z \psi(t)\|_{L^2} \big)^2, \\
\end{split}
\end{equation*}
where $E_0^\ep$ is defined in \eqref{eep0}. Therefore, we reduce estimates of $\| \psi(t) \|_{\Sigma^1}$ to  $\| z \psi(t) \|_{L^2}$. 
\begin{equation}\label{estz}
\begin{split}
 \f{d}{dt}\| z \psi(t)\|_{L^2}^2&=2\text{Re}\langle  z\psi(t) , z \partial_t \psi(t) \rangle \\
 &=-2\text{Im}\langle z^2\psi(t), \cH_z\psi(t) \rangle \\
 &= \text{Im}\langle[ \partial_z^2, z^2]\psi(t), \psi(t) \rangle \\
 &= \text{Im}\langle (1 +2z \partial_z)\psi(t) , \psi(t) \rangle \\
 &\lesssim \| z\psi(t)\|_{L^2}\| \partial_z \psi(t) \|_{L^2}.\\
\end{split}
\end{equation}
Because
\[ \| \partial_z \psi(t) \|_{L^2} \le \| \psi(t) \|_{\Sigma^1} \lesssim  | E_0^\ep[\psi_0] + cM[\psi_0]|^\h+ c^{\h}\| z \psi(t)\|_{L^2} \]
we have 
\[  \f{d}{dt}\| z \psi(t)\|_{L^2}^2 \lesssim |E_0^\ep[\psi_0] + cM[\psi_0]| + c\| z\psi(t) \|_{L^2}^2. \]
By Gronwall's lemma, 
\[ \| z\psi(t)\|_{L^2} \le  \big( \f{1}{c}|E_0^\ep[\psi_0] + cM[\psi_0]| + \| z\psi_0 \|_{L^2}^2\big)^\h e^{Cc|t|},  \]
which is followed by
\begin{equation}\label{estSigma}
 \| \psi(t) \|_{\Sigma^1} \lesssim  |E_0^\ep[\psi_0]+cM[\psi_0]|^\h +  (|E_0^\ep[\psi_0] + cM[\psi_0]| + c\| z\psi_0 \|_{L^2}^2\big)^\h e^{Cc|t|} .
\end{equation}
Hence, the solution is global.\\

\noindent
\underline{Case 2: $\lmd=-1$ and $V$ is bounded below}\\

Note that if $V$ is bounded below, we can assume $V$ is non-negative without loss of generality. Indeed, suppose $\psi(t)$ is the solution to \eqref{ep} and $c:= \inf_{z\in \R}V(z) \in (-\infty, 0)$, $\psi_c(t):=e^{ict}\psi(t)$ solves

\begin{equation}
i\partial_t \psi=\f{1}{\ep^2}\cH \psi+\cH_z\psi -c\psi +\lmd |\psi|^{2\sigma} \psi.
\end{equation}
Hence, in Case 2, we assume $V$ is non-negative.\\

By H\"{o}lder's  and  Gagliardo-Nirenberg's inequalities,
\[ \| \psi \|_{L^4}^4 \lesssim \Big\| \|\nabla_x \psi\|_{L_x^2}^\h \|\psi\|_{L_x^2}^\h \Big\|_{L_z^4}^4 \lesssim \|\nabla_x \psi\|_{L^2}^2 \| \partial_z \psi \|_{L^2}  \| \psi \|_{L^2} \]
and we obtain
\begin{equation}\label{estE0}
\f{1}{\ep^2}K[\psi(t)]+B[\psi(t)]-C_*K[\psi(t)]B[\psi(t)]^\h M[\psi_0]^\h \le E_0^\ep[\psi_0]
\end{equation}
for some $C_*>0$. Let $ G[\psi]:=K[\psi]+\ep^2 B[\psi] $, one has
\[  \f{1}{\ep^2}G[\psi(t)]-\f{C^*}{\ep}G[\psi(t)]^{\f{3}{2}}M[\psi_0]^\h \le E_0^\ep[\psi_0] < \f{1}{\ep^2}K[\psi_0]+B[\psi_0], \]
that is
\begin{equation*}
\begin{split}
G[\psi(t)] < K[\psi_0]+\ep^2B[\psi_0] +\ep C_* M[\psi_0]^\h G[\psi(t)]^{\f{3}{2}} .\\
\end{split}
\end{equation*}
We consider the function
 \begin{equation*}
\begin{split}
f_\ep(y):= K[\psi_0]+\ep^2 B[\psi_0] + \ep C_* M[\psi_0]^\h y^{\f{3}{2}}-y \qquad\text{for} \quad y \in \R_{\ge0}. \\
\end{split}
\end{equation*}
If $\ep$ satisfies
\begin{equation}\label{asep}
 0<\ep \le \Big( \f{4}{27C^* M[\psi_0]^\h (K[\psi_0]+B[\psi_0])}\Big)^\h  , 
 \end{equation}
there exist  $0<X_1^\ep< X_2^\ep (X_j^\ep=X_j^\ep(K[\psi_0], B[\psi_0], M[\psi_0], \ep))$ such that the condition
\[ 0<y < X_1^\ep, \quad  X_2^\ep < y\]
is equivalent to $f_\ep(y)>0$. Because $f_\ep(0)=K[\psi_0]+\ep^2B[\psi_0]>0$, $f_\ep(\f{4}{9\ep^2 {C_*}^2 M[\psi_0]})<0$, and  $f$ is convex, 
\begin{equation}
 G[\psi_0]=K[\psi_0]+\ep^2B[\psi_0]< X_1^\ep < 3(K[\psi_0]+\ep^2 B[\psi_0])
\end{equation}
\begin{equation}\label{X2}
\f{4}{9\ep^2 {C_*}^2 M[\psi_0]} <X_2^\ep
\end{equation}
hold for any $\ep>0$ that satisfies (\ref{asep}). Especially, it holds
\begin{equation}
X_1^\ep= K[\psi_0] +O(\ep) ,\quad   X_2^\ep=\f{1}{\ep^2 C_*^2 M[\psi_0]} +O(1) \text{\quad as \quad } \ep\to+0.
\end{equation}
The former follows from the fact  $X_1^\ep$ is bounded in the condition (\ref{asep})  and satisfies $X_1^\ep=K[\psi_0]+\ep^2 B[\psi_0] + \ep C_* M[\psi_0]^\h {X_1^\ep}^{\f{3}{2}}$. On the other hand, if we set $X_\ep^2=Y/(\ep^2 C_*^2 M[\psi_0])$, $Y$ satisfies
\begin{equation}\label{eqY}
 Y = (\ep^2 C_*^2 M[\psi_0])(K[\psi_0]+\ep^2 B[\psi_0]) + Y^\f{3}{2}. 
 \end{equation}
If $\ep$ is sufficiently small, (\ref{eqY}) has two solutions $Y_1^\ep$ and $Y_2^\ep$ such that $Y_1^\ep \to +0$ and  $Y_2^\ep\to 1$ as $\ep\to+0$. By (\ref{X2}), $Y=Y_2^\ep$. Hence, by (\ref{eqY}), for sufficiently small $\ep$ one has
\begin{equation*}
\begin{split}
|X_2^\ep-\f{1}{\ep^2 C_*^2 M[\psi_0]}|= \f{|(1-(Y_2^\ep)^\h)(1+(Y_2^\ep)^\h)|}{\ep^2 C_*^2 M[\psi_0]} \le 4 (K[\psi_0]+\ep^2 B[\psi_0]).
\end{split}
\end{equation*}
  Because $G[\psi(t)]$ is continuous with respect to $t$ as long as $G[\psi(t)]$ exists, it follows that
\[ G[\psi(t)]< X_1^\ep. \]
Hence, the estimate of $\|\psi(t) \|_{\Sigma^1}$ comes down to $\|z \psi(t) \|_{L^2}$.  The remainder of the proof is the same as Case 1.\\

\noindent
\underline{Case 3: $\lmd=-1$ and $V$ is unbounded below}.\\

The conclusion follows from existence of a global solution to (\ref{lmt}) (Theorem \ref{main2}) and iteration of Proposition \ref{3}.
\end{proof}

\subsection{The case $\sigma=2$}

\begin{thm}\label{wpep2}
Let $\sigma=2$ and $V$ satisfy \eqref{asV}. For any data $ \psi_0 \in \Sigma^1$ and any $\ep>0$, there exist $T>0$ and a unique solution to (\ref{ep})  $ \psi^\ep \in C([0, T], \Sigma^1)\cap L^3([0, T],\Sigma_{x,z}^{6,2})$, depending continuously on $\psi_0$. Suppose $\Tmax^\ep \in (0,+\infty]$ is the maximal time  of existence. If $\lmd=+1$, $\Tmax^\ep \to +\infty$ as $\ep\to+0$.
\end{thm} 

\begin{proof} For $0< T\le 1$ and $a,b>0$ we define 
\begin{equation*}
\begin{split}
 M(a,b,T)=\{ \psi \in L^\infty([0,T],\Sigma^1) \cap L^3([0,T],\Sigma_{x,z}^{6,2}) :  \| \psi \|_{L^\infty([0,T], \Sigma^1)}\le a, \\
 \| \psi\|_{L^3([0,T],\Sigma_{x,z}^{6,2})} \le b \}.
\end{split}
\end{equation*}
Bound for $\Psi^\ep[\psi]$ is similar to the case $\sigma=1$.
Instead of (\ref{Psinl}),  we use the following estimate, which is obtained based on H\"{o}lder's and Gagliardo-Nirenberg-Sobolev's inequalities:
\begin{equation*}
\begin{split}
\| \nabla_x |\psi|^4\psi  \|_{L_t^{3/2}L_x^{6/5}L_z^2([0,T])} &\lesssim \Big\|  \|\nabla_x \psi\|_{L_z^2} \|\psi\|_{L_z^\infty}^4 \Big\|_{L_t^{3/2}L_x^{6/5}}\\
&\lesssim \Big\|  \|\nabla_x\psi\|_{L_z^2} \|\partial_z\psi\|_{L_z^2}\|\psi\|_{L_z^6}^3 \Big\|_{L_t^{3/2}L_x^{6/5}}\\
&\lesssim \|\nabla_x\psi\|_{L_t^3L_x^6L_z^2} \|\partial_z\psi\|_{L_t^3L_x^6L_z^2} \|\psi\|_{L_t^\infty L_{x,z}^6}^3 \\
&\lesssim \|\nabla_x \psi\|_{L_t^3L_x^6L_z^2} \|\partial_z\psi\|_{L_t^3L_x^6L_z^2} \|\nabla_\bfx \psi\|_{L_t^\infty L_\bfx^2}^3.
\end{split}
\end{equation*}
We bound the other derivatives or weights for $\Psi^\ep$ in the same manner. Then, we have 
\begin{equation}
\begin{split}
\| \Psi^\ep[\psi] \|_{L_t^\infty \Sigma^1([0,T])}&\le C_0\|\psi\|_{L_t^\infty\Sigma^1}+ C_1\|\psi\|_{L_t^\infty\Sigma^1([0,T])}^3  \|\psi\|_{L_t^3\Sigma_{x,z}^{6,2}([0,T])}^2\\
&\le  C_0\|\psi_0\|_{\Sigma^1} + C_1 a^3 b^2
\end{split}
\end{equation}
and
\begin{equation}
\begin{split}
\| \Psi^\ep[\psi] \|_{L_t^3 \Sigma_{x,z}^{6,2}([0,T])}&\le\|e^{-it(\cH/\ep^2+\cH_z)}\psi_0\|_{L_t^3\Sigma_{x,z}^{6,2}([0,T])}+ C_1\|\psi\|_{L_t^\infty\Sigma^1([0,T])}^3  \|\psi\|_{L_t^3\Sigma_{x,z}^{6,2}([0,T])}^2\\
&\le \|e^{-it(\cH/\ep^2+\cH_z)}\psi_0\|_{L_t^3\Sigma_{x,z}^{6,2}([0,T])} + C_1 a^3b^2
\end{split}
\end{equation}
for some $C_0$, $C_1\ge1$. If we choose $a=2C_0\|\psi_0\|_{\Sigma^1}$, 
\begin{equation*}
 b\le \min\{ \f{1}{2C_1 a^3}, \f{a}{2}  \}=\min\{ \f{1}{2^4C_1(C_0\|\psi_0\|_{\Sigma^1})^3}, C_0\|\psi_0\|_{\Sigma^1}  \}
 \end{equation*}
and $0<T\le1$ so that 
\begin{equation}\label{a2}
 \|e^{-it(\cH/\ep^2+\cH_z)}\psi_0\|_{L_t^3\Sigma_{x,z}^{6,2}([0,T])}\le \f{b}{2}, 
\end{equation}
$\Psi^\ep$ is a mapping on $M(a,b,T)$. The contraction property, uniqueness, and continuity statements
are easy consequences of the fixed point argument. Persistence of regularity is obtained in the same manner as that stated in Theorem \ref{wpep1}. \\

We next prove $\Tmax^\ep \to +\infty$ as $\ep\to+0$  in the $\lmd=+1$ case. By Strichartz's estimate,
\[ \|e^{-it(\cH/\ep^2+\cH_z)}\psi_0\|_{L_t^3\Sigma_{x,z}^{6,2}([0,T])}\le C_2(\ep^2+T)^{\f{1}{3}} \|\psi_0\|_{\Sigma^1} \]
holds for some $C_2\ge1$. Then, if we assume
\begin{equation}\label{assep}
\ep^2\le \h \Big(\f{1}{C_2} \min\{ \f{1}{2^4C_1C_0^3 \|\psi_0\|_{\Sigma^1}^4}, C_0 \} \Big)^3,
\end{equation}
we can set $a$, $ b $ as mentioned above and
\begin{equation}\label{Tec2}
\begin{split}
T \le  \h\Big( \f{1}{C_2} \min\{ \f{1}{2^4C_1C_0^3 \|\psi_0\|_{\Sigma^1}^4}, C_0 \}  \Big)^3,
\end{split}
\end{equation}
which satisfies (\ref{a2}). $T$ is determined by $\|\psi_0\|_{\Sigma^1}$. Hence, retaking $\ep$ sufficiently small and  iterating Proposition \ref{3},  the solution to (\ref{ep}) can be extended to any compact time interval $I $ that satisfies $0\in I$ and $I \subsetneq [0. \Tmax)$, where $\Tmax$ is the maximal time of the solution to (\ref{lmt}) with an initial data $\psi_0$. If $\lmd=+1$, $\Tmax =+\infty$ (See Theorem \ref{wplmt2}). Hence, we have the conclusion.\\
\end{proof}

\section{Well-posed results of  (\ref{lmt})}\label{Wplmt}

\subsection{The case $\sigma=1$}
\begin{thm}\label{wplmt1}
Let $\sigma=1$ and $V$ satisfy \eqref{asV}. For any data $ \psi_0 \in L_x^2\Sigma_z^1$ there exists a unique global solution to (\ref{lmt})  $\phi \in C([0,\infty),L_x^2 \Sigma_z^1)$, depending continuously on $\psi_0$.
\end{thm}

\begin{proof} (\ref{lmt}) is equivalent to the  integral equation 
\begin{equation}\label{lmt2}
\phi(t)=e^{-it\cH_z}\psi_0-i\lmd\int_0^t e^{-i(t-s)\cH_z}F_{\text{av}}(\phi(s))ds.
\end{equation}
For $a>0$ and $0 <T\le 1$, we define
\[ M(a,T):=\{ \phi \in L([0,T],L_x^2\Sigma_z^1) : \| \phi \|_{L^\infty([0,T], L_x^2\Sigma_z^1)} \le a\} \]
and
\begin{equation}\label{intlmt}
 \Phi[\phi]:= e^{-it\cH_z}\psi_0-i\lmd\int_0^t e^{-i(t-s)\cH_z}F_{\text{av}}(\phi(s))ds .
\end{equation}
First, by Lemma \ref{sigmaz}, for some $C_0\ge1$ we have
\begin{equation}\label{ldu}
\begin{split}
 \| \Phi[\phi] \|_{L_t^\infty L_x^2 \Sigma_z^1([0,T])} \le C_0\| \psi_0 \|_{L_x^2 \Sigma_z^1}+C\Big\| \int_0^t \| F_{\text{av}}(\phi(s))  \|_{L_x^2 \Sigma_z^1} ds\Big\|_{L_t^\infty([0,T])}.
\end{split}
\end{equation}
We bound the second term on the right-hand side.  By commutativity of $e^{i\theta\cH}$ and $\partial_z$, Strichartz's estimate for $\{e^{i\theta\cH}\}_{\theta\in \R}$, Minkowski's inequality, and Gagliardo-Nirenberg's inequality, we have
\begin{equation}\label{reg1}
\begin{split}
\|\partial_z F_{\text{av}}(\phi(s)) \|_{L_\bfx^2} &\lesssim \| \partial_z (|e^{-i \theta \cH}\phi(t)|^2 e^{-i\theta\cH}\phi(t) )\|_{L_z^2 L_{\theta,x}^{4/3}(\R\times [0,2\pi]\times \R^2)}\\
&\lesssim \Big\| \| e^{-i \theta \cH} \partial_z\phi(t) \|_{L_{\theta,x}^4}  \| e^{-i \theta \cH} \phi(t) \|_{L_{\theta,x}^4}^2 \Big\|_{L_z^2}\\ 
&\lesssim \Big\| \| \partial_z\phi \|_{L_x^2} \|\phi \|_{L_x^2}^2 \Big\|_{L_z^2} \\
&\lesssim \| \partial_z \phi \|_{L_\bfx^2}  \| \partial_z \phi \|_{L_\bfx^2}  \|\phi \|_{L_\bfx^2}
\end{split}
\end{equation}
and a similar estimate holds if $\partial_z$ is replaced by $z$. Therefore, 
\begin{equation}\label{Favbdd1}
\begin{split}
\| F_{\text{av}}(\phi(s)) \|_{L_x^2 \Sigma_z^1} \lesssim \|\phi \|_{L_x^2\Sigma_z^1} \|\partial_z \phi \|_{L^2} \| \phi \|_{L^2} .\\
\end{split}
\end{equation}
consequently, we have
\[  \| \Phi[\phi] \|_{L_t^\infty L_x^2 \Sigma_z^1([0,T])} \le C_0 \| \psi_0 \|_{L_x^2 \Sigma_z^1} +C\| \phi \|_{L_t^\infty L_x^2 \Sigma_z^1([0,T])}^3T.\]
Thus, if we choose 
\begin{equation}\label{T0}
 a = 2C_0\| \psi_0 \|_{L_x^2 \Sigma_z^1} \qquad T\le \f{1}{Ca^2},
\end{equation}
 $\Phi$ is a mapping on $M(a,T)$.
Contraction, uniqueness, and continuous properties are also shown by the standard arguments.\\

We estimate $\|\phi(t)\|_{L_x^2\Sigma_z^1}$ to prove that the solution is global. We consider the case $\lmd=-1$ because this covers the case  $\lmd=+1$ as well.
Because $V(z)$ is sub-quadratic, there exists $c\ge 1$ such that
\[ 0 \le V(z)+c(1+z^2) \simeq \langle z \rangle^2 . \]

We  first observe the nonlinear term of energy $E$ is bounded by Strichartz's estimate and Gagliardo-Nirenberg's inequality as follows,
\begin{equation}
\int_{\R^3}\int_0^{2\pi} |e^{-i\theta\cH}\phi|^{4}d\theta d\bfx \lesssim \|\phi\|_{L^2}^3\| \partial_z \phi \|_{L^2}. 
\end{equation}
Then, for some $C_*=C_*(\sigma)>0$ one has
\[ B_1[\phi(t)]+B_2[\phi(t)]-C_*M[\psi_0]^{\f{3}{2}}B_1[\phi(t)]^{\h}\le E[\psi_0]\]
where, 
\[  B_1[\phi]:= \h \| \partial_z \phi \|_{L^2}^2 \qquad   B_2[\phi] := \int _{\R^3} V(z) |\phi|^2 d\bfx .  \]
We estimate $B_1[\phi(t)]$. If $ B_1[\phi(t)] \le C_*^2M[\psi_0]^3$ holds, we already have the bound for $B_1[\phi(t)]$ and obtain 
\begin{equation*}
\begin{split}
\| \phi(t) \|_{L_x^2 \Sigma_z^1}^2 &\lesssim B_1[\phi(t)] + B_2[\phi(t)] +cM[\phi(t)]+c\| z \phi(t)\|_{L^2}^2  \\
&\le  E[\psi_0] + C_*M[\psi_0]^{\f{3}{2}}B_1[\phi(t)]^{\h} +cM[\psi_0]+c\| z \phi(t)\|_{L^2}^2 \\
&\le E[\psi_0] +  C_*^2M[\psi_0]^3 +cM[\psi_0]+c\| z \phi(t)\|_{L^2}^2.
\end{split}
\end{equation*}
Otherwise, 
\begin{equation*}
\begin{split}
B_1[\phi(t)] &\le  4(B_1[\phi(t)]^\h - \f{ C_*M[\psi_0]^{\f{3}{2}}}{2})^2\\
&\le 4\big(B_1[\phi(t)]-C_*M[\psi_0]^{\f{3}{2}}B_1[\phi(t)]^{\h} \big) + C_*^2M[\psi_0]^3\\
&\le4(E[\psi_0] -B_2[\phi(t)] )  +C_*^2M[\psi_0]^3,
\end{split}
\end{equation*}
that is 
\begin{equation*}
\begin{split}
B_1[\phi(t)]+4B_2[\phi(t)]\le4E[\psi_0] +C_*^2 M[\psi_0]^3.
\end{split}
\end{equation*}
Thus, we have
\begin{equation*}
\begin{split}
\| \phi(t) \|_{L_x^2 \Sigma_z^1}^2 &\lesssim B_1[\phi(t)] + 4(B_2[\phi(t)] + cM[\phi(t)]+c\| z \phi(t)\|_{L^2}^2)  \\
&\le 4E[\psi_0] +  C_*^2M[\psi_0]^3 + 4cM[\psi_0]+4c\| z \phi(t)\|_{L^2}^2.
\end{split}
\end{equation*}
If $\lmd=+1$, it holds  
\begin{equation*}
\begin{split}
\| \phi(t) \|_{L_x^2 \Sigma_z^1}^2 \lesssim E[\psi_0] + c M[\psi_0]+ c \| z \phi(t)\|_{L^2}^2. 
\end{split}
\end{equation*}
The remainder of the proof is similar to  the argument presented in (\ref{estz}) and (\ref{estSigma}). Finally, we have
\begin{equation*}
\| z\phi(t) \|_{L^2} \le \big( C(M[\psi_0],E[\psi_0],c) + \|z \psi_0 \|_{L^2}^2 \big)^\h e^{Cc|t|},
\end{equation*}
which implies the solution is global.\\

To obtain persistence of regularity, we need
\begin{equation}\label{perlmt1}
 \| F_{\text{av}}(\phi) \|_{\Sigma^k} \lesssim \| \phi \|_{\Sigma^k} \| \phi \|_{L_x^2\Sigma_z^1}  \|\phi\|_{L^2}. 
 \end{equation}
This follows from (\ref{eqH}), same calculation as that performed in  (\ref{reg1}), and interpolation. 
\end{proof}

\subsection{The case $\sigma=2$}

\begin{thm}\label{wplmt2}
Let $\sigma=2$ and $V$ satisfy \eqref{asV}. For any data  $\psi_0 \in  \Sigma^1$, there exists $T>0$ and  a unique solution to (\ref{lmt})  $ \phi\in C([0,T],\Sigma^1)$, depending continuously on $\psi_0$.  If $\lmd=+1$, the solution is global.
\end{thm}

Basically, the proof is the same as that for the case $\sigma=1$. We bound the nonlinear term of $\Phi[\phi]$ as follows.
For the $x$-derivatives and weights, applying (\ref{eqH}), Strichartz's estimate, Minkowski's inequality, and Gagliardo-Nirenberg's inequality, we have
\begin{equation*}
\begin{split}
\| \cH_0^\h \Fav(\phi(t)) \|_{L_{\bfx}^2} &\lesssim \| \cH_0^{\h}(|e^{-i\theta\cH}\phi(t)|^4e^{-i\theta\cH}\phi(t)) \|_{L_z^2L_{\theta}^{3/2}L_x^{6/5}(\R\times[0,2\pi]\times \R^2)}\\
&\lesssim \Big\| \| \cH_0^{\h}e^{-i\theta\cH}\phi(t) \|_{L_x^6} \| e^{-i\theta\cH}\phi(t)\|_{L_x^6}^4 \Big\|_{L_{\theta}^{3/2}L_z^2}\\
&\lesssim \| e^{-i\theta\cH} \cH_0^{\h}\phi(t)\|_{L_{\theta}^3L_z^2L_x^6} \|e^{-i\theta\cH}\phi(t) \|_{L_{\theta}^{12}L_x^{6}L_z^{\infty}}^4\\
&\lesssim \| e^{-i\theta\cH} \cH_0^{\h}\phi(t)\|_{L_z^2L_{\theta}^3L_x^6} \|\partial_z e^{-i\theta\cH}\phi(t) \|_{L_{\theta}^{3}L_x^6L_z^2}\|e^{-i\theta\cH}\phi(t) \|_{L_{\theta}^{\infty}L_{x,z}^6}^3\\
\end{split}
\end{equation*}
 For the first two elements, using Minkowski and Strichartz for $L_\theta^3L_x^6$, and for the last element, using Sobolev's embedding $ \Sigma^1 \hookrightarrow L_\bfx^6$, we obtain
\begin{equation*}
\| \cH_0^{\h} \Fav(\phi(t)) \|_{L_{\bfx}^2} \lesssim \|\phi(t)\|_{\Sigma^1}^5 .
\end{equation*}
Estimating in the same manner for the $z$-derivative and corresponding weight, we have
 \begin{equation*}
\| \Fav(\phi(t)) \|_{\Sigma^1} \lesssim \|\phi(t)\|_{\Sigma^1}^5 .
\end{equation*}
By interpolation, for any integer $k\in\N$ we also have
\begin{equation}\label{Favbdd2}
\| \Fav(\phi(t)) \|_{\Sigma^k} \lesssim \|\phi(t)\|_{\Sigma^k} \|\phi(t)\|_{\Sigma^1}^4.
\end{equation}

Considering the bound of $\|\phi(t)\|_{\Sigma^1}$ in the case $\lmd=+1$,   the $x$-derivatives and corresponding weights of $\phi(t)$ are controlled by $K[\phi]$ expressed in \eqref{ke}, and the $z$-derivative and weight can be treated  as the case  $\sigma=1$. \\

\subsection{The case $\sigma=3$}

\begin{thm}
Let $\sigma=3$ and $V$ satisfy \eqref{asV}. For any data  $\psi_0 \in  \Sigma^1$, there exists $T>0$ and  a unique solution to (\ref{lmt})  $ \phi\in C([0,T],\Sigma^1)\cap L^{12}([0,T], \Sigma_{z,x}^{3,2})$, depending continuously on $\psi_0$.  If $\lmd=+1$, the solution is global.

\end{thm}

\noindent
When $\sigma=3$, we need the $L_z^p$ norm for $p>2$, thus, we cannot use Lemma \ref{sigmaz}.

\begin{proof}
For $a>0$ and $0< T\le 1$, we define
\[ M(a,T):=\{ \phi \in L^\infty([0,T],\Sigma^1)\cap L^{12}([0,T],\Sigma_{z,x}^{3,2}) : \| \phi \|_{L_t^\infty\Sigma^1\cap L_t^{12}\Sigma_{z,x}^{3,2}([0,T])} \le a\} \]
and $\Phi$ as (\ref{intlmt}). We first bound $x$-derivatives and weights of $\Phi[\phi]$ as
\begin{equation*}
\begin{split}
 \| \nabla_x \Phi[\phi]\|_{L^\infty_t L_\bfx^2\cap L_t^{12}L_z^3 L_x^2([0,T])}+ \| x \Phi[\phi]\|_{L^\infty_t L_\bfx^2\cap L_t^{12}L_z^3 L_x^2([0,T])} \\
 \le C_0\|\phi_0\|_{\Sigma^1}+C\|\cH_0^{\h}\Fav(\phi)\|_{L_t^1 L_\bfx^2([0,T])} 
\end{split}
\end{equation*}
and 
\begin{equation}\label{31}
\begin{split}
\|\cH_0^\h\Fav(\phi(t))\|_{L_\bfx^2}&\lesssim \| \cH_0^{\h}(|e^{-i\theta\cH}\phi(t)|^6e^{-i\theta\cH}\phi(t)) \|_{L_{\theta}^1L_\bfx^2([0,2\pi]\times\R^3)}\\
&\lesssim \| \cH_0^{\h}e^{-i\theta\cH}\phi(t)\|_{L_{\theta}^6L_\bfx^3} \|e^{-i\theta\cH}\phi(t) \|_{L_\theta^{\f{36}{5}} L_\bfx^{36}}^6\\
&\lesssim \| \phi(t)\|_{\Sigma_{z,x}^{3,2}}  \|e^{-i\theta\cH}\phi(t) \|_{L_\theta^{\f{36}{5}} L_\bfx^{36}}^6\\
\end{split}
\end{equation}
We focus on the latter factor. Then, we have
\begin{equation}\label{32}
\begin{split}
\|e^{-i\theta\cH}\phi(t) \|_{L_\theta^{\f{36}{5}}L_\bfx^{36}}&\lesssim  \Big\| \| \nabla_\bfx e^{-i\theta\cH}\phi(t) \|_{L_\bfx^3}^{\f{5}{6}} \|e^{-i\theta\cH}\phi(t) \|_{L_\bfx^6}^{\f{1}{6}} \Big\|_{L_\theta^{\f{36}{5}}}\\
&\lesssim \|  e^{-i\theta\cH} \cH_0^\h \phi(t) \|_{L_\theta^6 L_\bfx^3}^{\f{5}{6}}  \| e^{-i\theta\cH}\phi(t) \|_{L_\theta^\infty L_\bfx^6}^{\f{1}{6}}\\
&\lesssim \| e^{-i\theta\cH}\cH_0^\h \phi(t) \|_{L_\theta^6 L_\bfx^3}^{\f{5}{6}} \| \nabla_\bfx e^{-i\theta\cH}\phi(t) \|_{L_\theta^\infty L_\bfx^2}^{\f{1}{6}}\\
&\lesssim \| \phi(t)\|_{\Sigma_{z,x}^{3,2}}^{\f{5}{6}} \|\phi(t)\|_{\Sigma^1}^{\f{1}{6}}.
\end{split}
\end{equation}
From these estimates,
\begin{equation*}
\begin{split}
\|\cH_0^{\h}\Fav(\phi)\|_{L_t^1 L_\bfx^2([0,T])}&\lesssim \| \phi\|_{L_t^{12}\Sigma_{z,x}^{3,2}([0,T])}^6 \|\phi\|_{L_t^\infty \Sigma^1([0,T])}T^\h.
\end{split}
\end{equation*}
For bounding $\partial_z \Phi[\phi]$ and $z\Phi[\phi]$, we first differentiate $\Phi[\phi(t)]$ with respect to $t$: 
\[ i\partial_t \Phi[\phi(t)]= \cH_z \Phi[\phi(t)] +\lmd \Fav(\phi(t))   \qquad \Phi[\phi(0)]=\psi_0. \]
By differentiating both sides with respect to $z$ or multiplying by $z$, and applying Duhamel's formula again, we obtain
\begin{equation}\label{lmt2-2}
\partial_z \Phi[\phi(t)]=e^{-it\cH_z}\partial_z \psi_0-i\int_0^t e^{-i(t-s)\cH_z} [\lmd\partial_z  F_{\text{av}}(\phi(s))+V'(z)\Phi[\phi(s)]]ds
\end{equation}
\begin{equation}\label{lmt2-3}
z\Phi[\phi(t)]=e^{-it\cH_z}z\psi_0-i\int_0^t e^{-i(t-s)\cH_z} [\lmd z  F_{\text{av}}(\phi(s))+\partial_z \Phi[\phi(s)]]ds.
\end{equation}
We use (\ref{lmt2-2}) and (\ref{lmt2-3}) to obtain
\begin{equation*}
\begin{split}
\|\partial_z \Phi[\phi]\|_{L_t^\infty L_\bfx^2 \cap L_t^{12} L_z^3L_x^2} &\le C_0\|\psi_0\|_{\Sigma^1}+C\|\partial_z \Fav(\phi)\|_{L_t^1 L_\bfx^2} +C\| V'(z)\Phi[\phi]\|_{L_t^1 L_\bfx^2} \\
 &\le C_0\|\psi_0\|_{\Sigma^1}+C \| \phi\|_{L_t^{12}\Sigma_{z,x}^{3,2}}^6 \|\phi\|_{L_t^\infty \Sigma^1}T^\h  +C\| \langle z \rangle \Phi[\phi]\|_{L_t^\infty L_\bfx^2}T
\end{split}
\end{equation*}
and
\begin{equation*}
\begin{split}
\|  z  \Phi[\phi]\|_{L_t^\infty L_\bfx^2 \cap L_t^{12} L_z^3L_x^2} &\le C_0\|\psi_0\|_{\Sigma^1}+C \| \phi\|_{L_t^{12}\Sigma_{z,x}^{3,2}}^6 \|\phi\|_{L_t^\infty \Sigma^1}T^\h +C\| \partial_z \Phi[\phi]\|_{L_t^\infty L_\bfx^2} T.
\end{split}
\end{equation*}
Summing up these estimates, for some $C_1\ge 1$, we have
\begin{equation*}
\begin{split}
\| \partial_z \Phi[\phi]  \|_{L_t^\infty L_\bfx^2 \cap L_t^{12} L_z^3  L_x^2}& + \|  z   \Phi[\phi]\|_{L_t^\infty L_\bfx^2 \cap L_t^{12} L_z^3L_x^2}\\
&\le 2C_0\|\psi_0\|_{\Sigma^1}+C \| \phi\|_{L_t^{12}\Sigma_{z,x}^{3,2}}^6 \|\phi\|_{L_t^\infty \Sigma^1}T^\h  \\
& \quad +C_1\big(\| \partial_z \Phi[\phi]\|_{L_t^\infty L_\bfx^2 \cap L_t^{12} L_z^3L_x^2} + \|  z  \Phi[\phi]\|_{L_t^\infty L_\bfx^2 \cap L_t^{12} L_z^3L_x^2} \big)T
\end{split}
\end{equation*}
where, note that 
\[ \| \partial_z u\|_{L_z^2}^2+\|\langle z \rangle u\|_{L_z^2}^2\simeq \|\partial_z u\|_{L_z^2}^2+\| z u\|_{L_z^2}^2. \]
Hence, if $T\le (2C_1)^{-1}$, we obtain 
\begin{equation*}
\begin{split}
\| \partial_z \Phi[\phi]\|_{L_t^\infty L_\bfx^2 \cap L_t^{12} L_z^3L_x^2} + \|  z   \Phi[\phi]\|_{L_t^\infty L_\bfx^2 \cap L_t^{12} L_z^3L_x^2} \le 4C_0\|\psi_0\|_{\Sigma^1}+C \| \phi\|_{L_t^{12}\Sigma_{z,x}^{3,2}}^6 \|\phi\|_{L_t^\infty \Sigma^1}T^\h .\\
\end{split}
\end{equation*}
Finally, for some $C_2\ge1$, we have
\begin{equation*}
\begin{split}
\|\Phi[\phi]\|_{L_t^\infty \Sigma^1\cap  L^\infty\Sigma_{z,x}^{3,2}([0,T])}&\le C_2 \|\phi_0\|_{\Sigma^1}+C\| \phi\|_{L_t^{12}\Sigma_{z,x}^{3,2}([0,T])}^6 \|\phi\|_{L_t^\infty \Sigma^1([0,T])}T^\h .\\
\end{split}
\end{equation*}
By choosing $a=2C_2\|\psi_0\|_{\Sigma^1}$ and $T\le \min\{ (2Ca^6)^{-2},(2C_1)^{-1} \}$, $\Phi$ is on  $M(a,T)$.
The remainder of the proof (the case $\lmd=+1$) is the same as that for the cases $\sigma=1,2$.\\

To obtain the  persistence of regularity, in (\ref{31}) and (\ref{32}) we replace $\cH_0$ by 3-dimensional harmonic oscillator $\cH_{0,\bfx}=-\h\Delta_\bfx + \f{1}{8}|\bfx|^2$ and apply (\ref{eqH}) in all 3 dimensions. Then, for any integer $k\in \N$, we have
\begin{equation*}
\begin{split}
\|\Fav(\phi)\|_{L_t^1 \Sigma^k([0,T])}&\lesssim \| \phi\|_{L_t^{12}\Sigma_{z,x}^{3,2,k}([0,T])}\| \phi\|_{L_t^{12}\Sigma_{z,x}^{3,2}([0,T])}^5 \|\phi\|_{L_t^\infty \Sigma^1([0,T])}T^\h.
\end{split}
\end{equation*}

\end{proof}

\subsection{The case $\sigma=4$ (energy-critical) }
We denote
\[ \|u\|_{\dot{X}_x([-T,T])}:= \|\nabla_x u\|_{L_t^\infty L_\bfx^2 \cap  L_t^{4}L_z^{\infty} L_x^2([-T,T]\times\R^3)}+\||x|u\|_{L_t^\infty L_\bfx^2 \cap L_t^{4} L_z^{\infty} L_x^2([-T,T]\times\R^3)} ,\]

\[ \|u\|_{\dot{X}_z([-T,T])}:= \|\partial_z u\|_{L_t^\infty L_\bfx^2\cap  L_t^{4}L_z^{\infty} L_x^2([-T,T]\times\R^3)}+\|zu\|_{L_t^\infty L_\bfx^2 \cap L_t^{4} L_z^{\infty} L_x^2([-T,T]\times\R^3)} .\]
\begin{thm}\label{wplmt4}
Let $\sigma=4$ and $V$ satisfy \eqref{asV}. There exist absolute constants $\al>0$ and $T_0\in (0,1]$ such that the following holds. \\

If $\psi_0\in \Sigma^1$ and  $T_1>0$ satisfy
\begin{equation}\label{ec}
\begin{split}
 \|e^{-it\cH_z} \psi_0\|_{L_t^4L_z^\infty \Sigma_x^1 ([-T_1,T_1])}^{\f{1}{4}} \|\psi_0\|_{ L_z^2\Sigma_x^1 }^{\f{1}{2}} \| \psi_0\|_{L_x^2\Sigma_z^1}^{\f{1}{4}} \le \al,
\end{split}
\end{equation}
then for $T\le \min\{T_0, T_1\}$, (\ref{lmt}) has a unique solution $\phi \in C_t\Sigma^1\cap L_t^4  \Sigma_{z,x}^{\infty,2}([-T,T]\times\R^3)$, depending continuously on $\psi_0$.\\

\end{thm}

\begin{proof}
Note that for any $\al>0$ and $\psi_0\in \Sigma^1$, there exists $T_1>0$ which satisfies (\ref{ec}).
Fix $\psi_0 \in \Sigma^1$.  For $\al>0$, $\beta= (\beta_1,\beta_2,\beta_3) \in \R_{>0}^3$ and $0<T\le1$ we define

\begin{equation*}
\begin{split}
 M_\al(\beta,T):=\{ \phi \in L_t^\infty\Sigma^1\cap L_t^4\Sigma_{z,x}^{\infty,2}&([-T,T]) : \|\phi \|_{L_t^4L_z^\infty \Sigma_x^1([-T,T])} \le \beta_1, \\
 &  \|\phi\|_{\Dot{X}_x([-T,T])}  \le \beta_2, \text{ } \|\phi\|_{\Dot{X}_z([-T,T])} \le \beta_3   \},
\end{split}
\end{equation*}
and $\Phi$ as (\ref{intlmt}). We prove the map $\Phi$ is on $M_\al(\beta,T)$ and it is a contraction for appropriate $\al$, $\beta$ and $T$. First, by Strichartz's estimate, we have
\begin{equation}\label{11}
\begin{split}
 \| \Phi[\phi]\|_{\Dot{X}_x([-T,T])}  \lesssim \|\psi_0\|_{L_z^2\Sigma_x^1}+\|\cH_0^{\h}\Fav(\phi)\|_{L_x^2 L_t^{\f{4}{3}} L_z^1} .
\end{split}
\end{equation}
On the nonlinear term, we use Minkowski's inequality, Strichartz's estimate, and H\"{o}lder's inequality: 
\begin{equation}\label{12}
\begin{split}
 \|\cH_0^\h \Fav(\phi) \|_{L_t^{\f{4}{3}}L_z^1L_x^2} &\le \Big\| \|\cH_0^\h(|e^{-i\theta\cH}\phi|^8 e^{-i\theta\cH}\phi)\|_{L_{\theta,x}^{\f{4}{3}}} \Big\|_{L_t^{1}L_z^2}\\
&\lesssim  \Big\| \|\cH_0^\h e^{-i\theta\cH}\phi \|_{L_{\theta,x}^4} \|(e^{-i\theta\cH}\phi)^8\|_{L_{\theta,x}^2} \Big\|_{L_t^{1}L_z^2}\\
&\lesssim \|\cH_0^\h \phi\|_{L_t^{4} L_z^\infty L_{x}^2} \|e^{-i\theta\cH}\phi\|_{L_t^{\f{32}{3}}L_z^{16}L_{\theta}^{16} L_x^{16}}^8 .\\
\end{split}
\end{equation}
By Gagliardo-Nirenberg's, H\"{o}lder's, and Minkowski's inequalities, and Strichartz's estimate, 
\begin{equation}\label{13}
\begin{split}
 \| e^{-i\theta\cH}\phi\|_{L_t^{\f{32}{3}}L_z^{16} L_{\theta}^{16} L_x^{16}} &\lesssim \Big\| \|\nabla_x e^{-i\theta\cH}\phi\|_{L_x^2}^{\f{1}{4}} \| e^{-i\theta\cH}\phi \|_{L_x^{12}}^{\f{3}{4}} \Big\|_{L_t^{\f{32}{3}} L_z^{16} L_\theta^{16}}\\
 &\lesssim \| \cH_0^\h e^{-i\theta\cH}\phi\|_{L_t^{4}L_z^{\infty} L_{\theta}^{\infty}L_x^{2}}^{\f{1}{4}}  \| e^{-i\theta\cH}\phi \|_{L_{t}^{24} L_\theta^{12} L_\bfx^{12}}^{\f{3}{4}} \\
 &\lesssim \| \cH_0^\h \phi\|_{L_t^{4}L_z^\infty L_x^2}^{\f{1}{4}}  \| \nabla_x e^{-i\theta\cH}\phi \|_{L_{t}^{24} L_\theta^{12} L_\bfx^{\f{12}{5}} }^{\f{1}{2}} \| \partial_z e^{-i\theta\cH}\phi \|_{L_{t}^{24} L_\theta^{12} L_\bfx^{\f{12}{5}} }^{\f{1}{4}}  \\
 &\lesssim \| \phi\|_{L_t^{4}L_z^\infty \Sigma_x^1}^{\f{1}{4}}  \| \phi \|_{\Dot{X}_x}^{\f{1}{2}} \| \phi \|_{\Dot{X}_z}^{\f{1}{4}}.
\end{split}
\end{equation}
From these bounds, we have 
\begin{equation*}
\begin{split}
\| \Phi[\phi]\|_{\Dot{X}_x^1([-T,T])} &\le C_0\|\psi_0\|_{L_z^2\Sigma_x^1}+C_1 \|  \phi\|_{L_t^{4}L_z^\infty \Sigma_x^1}^{2}  \| \phi \|_{\Dot{X}_x}^{5} \| \phi \|_{\Dot{X}_z}^{2}\\
 &\le C_0 \|\psi_0\|_{L_z^2\Sigma_x^1}+C_1\beta_1^2\beta_2^5\beta_3^2
\end{split}
\end{equation*}
for some $C_0, C_1 \ge1$. Similarly,  we also have
\begin{equation}\label{14}
\begin{split}
 \|  \Phi[\phi]\|_{L_t^4 L_z^\infty \Sigma_x^1([-T,T])}  \le \|e^{-it\cH_z} \psi_0\|_{L_t^4L_z^\infty \Sigma_x^1([-T,T])}+C_1\beta_1^3\beta_2^4\beta_3^2.
\end{split}
\end{equation}
For the $z$-derivative and weight, we use (\ref{lmt2-2}) and (\ref{lmt2-3}) and obtain the following estimate for some $\tilde{C}_1\ge1$,
\begin{equation*}
\begin{split}
 \|  \Phi[\phi]\|_{\Dot{X}_z ([-T,T])} \le & C\|\psi_0\|_{L_x^2\Sigma_z^1}+C\| \phi\|_{L_t^\infty L_x^2 \Sigma_z^1([-T,T])} \|e^{-i\theta\cH}\phi\|_{L_t^{\f{32}{3}}L_z^{16}L_{\theta}^{16} L_x^{16}}^8 \\
 &+ \tilde{C}_1 \| \Phi[\phi]\|_{L_t^\infty L_x^2 \Sigma_z^1([-T,T])}T.
\end{split}
\end{equation*}
If $T \le T_0:=1/(2\tilde{C}_1) $, we obtain 
\begin{equation*}
\begin{split}
  \| \Phi[\phi]\|_{\Dot{X}_z ([-T,T])} \le C_0 \|\psi_0\|_{L_x^2\Sigma_z^1}+ C_1\beta_1^2\beta_2^4\beta_3^3.
\end{split}
\end{equation*}
We choose $\al$ so that 
\[ 1+2^{9}C_0^6C_1\al^8  \le2 \]
and $T_1>0$, which satisfies (\ref{ec}). Moreover, we choose 
 \[ \beta_1= 2\| e^{-it\cH_z} \psi_0\|_{L_t^4L_z^\infty \Sigma_x^1([-T_1.T_1])}, \quad  \beta_2= 2C_0\|\psi_0\|_{L_z^2\Sigma_x^1}, \quad \beta_3= 2C_0\|\psi_0\|_{L_x^2\Sigma_z^1}, \] 
and  $T\le \min\{T_0, T_1\}$, then the mapping $\Phi$ is on $M_\al(\beta,T)$.\\

For the contraction,  modifying (\ref{11}) through (\ref{13}), we have the following bound:
\begin{equation*}
\begin{split}
\| \Phi[\phi_1]-\Phi[\phi_2] \|_{\Dot{X}_x} &\lesssim\|\cH_0^{\h}(\Fav(\phi_1)-\Fav(\phi_2))\|_{L_x^2 L_t^{\f{8}{7}} L_z^{\f{4}{3}}}\\
&\lesssim \|\cH_0^\h(\phi_1-\phi_2)\|_{L_t^8 L_\theta^4 L_{\bfx}^4} (\|\phi_1\|_{L_t^{\f{32}{3}}L_\theta^{16} L_\bfx^{16}}^8+ \|\phi_2\|_{L_t^{\f{32}{3}}L_\theta^{16} L_\bfx^{16}}^8 ) \\
&\quad +(\|\cH_0^\h\phi_1 \|_{L_t^8 L_\theta^4 L_{\bfx}^4} +   \| \cH_0^\h \phi_2\|_{L_t^8 L_\theta^4 L_{\bfx}^4})\|\phi_1-\phi_2\|_{L_t^{24} L_\theta^{12} L_{\bfx}^{12}} \\
&\qquad \times (\|\phi_1 \|_{L_t^{\f{168}{17}} L_\theta^{\f{84}{5}} L_{\bfx}^{\f{84}{5}}}^7 + \|\phi_2\|_{L_t^{\f{168}{17}} L_\theta^{\f{84}{5}} L_{\bfx}^{\f{84}{5}}}^7) \\
&\lesssim \|\phi_1-\phi_2\|_{\Dot{X}_x([-T,T])}(\|\phi_1\|_{L_t^4 L_z^\infty \Sigma_x^1}+\|\phi_2\|_{L_t^4 L_z^\infty \Sigma_x^1})^2 \\
&\qquad \times (\|\phi_1\|_{\Dot{X}_x}+\|\phi_2\|_{\Dot{X}_x})^4(\|\phi_1\|_{\Dot{X}_z}+\|\phi_2\|_{\Dot{X}_z})^2\\
&\lesssim \|\phi_1-\phi_2\|_{\Dot{X}_x([-T,T])}\al^8 .
\end{split}
\end{equation*} 
Because we have a similar bound for the $\Dot{X}_z$ norm, for some $C_2 \ge 1$, we have
\begin{equation*}
\begin{split}
&\| \Phi[\phi_1]-\Phi[\phi_2] \|_{L_t^\infty \Sigma^1\cap L_t^4 \Sigma_{z,x}^{\infty,2}([-T,T])} \le C_2 \|\phi_1-\phi_2\|_{L_t^\infty \Sigma^1\cap L_t^4 \Sigma_{z,x}^{\infty,2}([-T,T])}\al^8 .
\end{split}
\end{equation*}
If we impose $C_2 \al^8\le 1/2$,  $\Phi$ becomes a contraction. Uniqueness is proved by almost the same (and standard) argument. \\

Next, we prove the continuous dependence.  Suppose that   $\phi_j \in  C_t\Sigma^1\cap L_t^4 \Sigma_{z,x}^{\infty ,2} ([-T_j,T_j]\times\R^3)$ is a solution with $\phi_j(0)=\psi_{0}^{(j)} \in \Sigma^1$ ($j=1,2$). We assume $\psi_0^{(2)}$ is in the neighborhood of $\psi_0^{(1)}$. Then, for any $\tilde{\alpha} \in (0,\alpha]$, there exists $\tilde{T} \in (0, \min\{ T_1,T_2,1\} ]$, which satisfies
 \begin{equation}
\begin{split}
&( \|e^{-it\cH_z} \psi_0^{(1)}\|_{L_t^4L_z^\infty \Sigma_x^1 ([-\tilde{T},\tilde{T}])} + \|e^{-it\cH_z} \psi_0^{(2)}\|_{L_t^4L_z^\infty \Sigma_x^1 ([-\tilde{T}, \tilde{T}])})^{\f{1}{4}}\\
&\quad \times (\|\psi_0^{(1)}\|_{ L_z^2\Sigma_x^1 }+ \|\psi_0^{(2)}\|_{ L_z^2\Sigma_x^1 })^{\f{1}{2}} (\| \psi_0^{(1)}\|_{L_x^2\Sigma_z^1} + \| \psi_0^{(2)}\|_{L_x^2\Sigma_z^1})^{\f{1}{4}} \le \tilde{\al}.
\end{split}
\end{equation}
Using the definition of $M_\alpha(\beta,T)$ for each initial data, we have
\begin{equation*}
\begin{split}
&\| \phi_1- \phi_2 \|_{L_t^\infty \Sigma^1\cap L_t^4 \Sigma_{z,x}^{\infty,2}([-\tilde{T},\tilde{T}])} \\
&\lesssim  \|\psi_0^{(1)}-\psi_0^{(2)}\|_{\Sigma^1} + \|\phi_1-\phi_2\|_{L_t^\infty \Sigma^1\cap L_t^4 \Sigma_{z,x}^{\infty,2}([-\tilde{T},\tilde{T}])}(\|\phi_1\|_{L_t^4 L_z^\infty \Sigma_x^1}+\|\phi_2\|_{L_t^4 L_z^\infty \Sigma_x^1})^2\\
&\qquad \times(\|\phi_1\|_{\Dot{X}_x}+\|\phi_2\|_{\Dot{X}_x})^4 (\|\phi_1\|_{\Dot{X}_z}+\|\phi_2\|_{\Dot{X}_z})^2\\
&\lesssim  \|\psi_0^{(1)}-\psi_0^{(2)}\|_{\Sigma^1}+  \|\phi_1-\phi_2\|_{L_t^\infty \Sigma^1\cap L_t^4 \Sigma_{z,x}^{\infty,2}([-\tilde{T},\tilde{T}])}\tilde{\al}^8 .
\end{split}
\end{equation*}
If we choose $\tilde{\al}$ sufficiently small, we obtain 
\begin{equation*}
\begin{split}
\| \phi_1- \phi_2 \|_{L_t^\infty \Sigma^1\cap L_t^4 \Sigma_{z,x}^{\infty,2} ([-\tilde{T},\tilde{T}])} \lesssim  \|\psi_0^{(1)}-\psi_0^{(2)}\|_{\Sigma^1}.
\end{split}
\end{equation*}
Iterating the above argument assuming $\psi_0^{(2)}$ is sufficiently closed to $\psi_0^{(1)}$, Lipschitz continuity in $L_t^\infty \Sigma^1\cap L^4 \Sigma_{z,x}^{\infty, 2}$ is proved.  Proof of Theorem \ref{wplmt4} is complete.
\end{proof}

\quad\\

We next prove Theorem \ref{scatter}. We first state the following modification of Theorem \ref{wplmt4} to fit the case of $V\equiv0$. Recall
\[ \|\phi\|_{\Sigma_0^1}=( \|\nabla_\bfx \phi\|_{L^2}^2 + \| |x|\phi \|_{L^2}^2)^\h.\]

\begin{thm}\label{wplmt40}
Let $\sigma=4$ and $V\equiv0$. There exists an absolute constant $\al>0$ such that the following holds. \\

If $\psi_0\in \Sigma_0^1$ and $T_1>0$ satisfy
\begin{equation}\label{ec0}
\begin{split}
\|e^{it\f{\partial_z^2}{2}} \psi_0\|_{L_t^4L_z^\infty \Sigma_x^1\cap L_t^{24}L_z^{\f{12}{5}} \Sigma_x^1 ([-T_1,T_1])}^{\f{3}{4}}\|\partial_z e^{it\f{\partial_z^2}{2}} \psi_0\|_{L_t^{24}L_z^{\f{12}{5}} L_x^2([-T_1,T_1])}^{\f{1}{4}} \le \al,
\end{split}
\end{equation}
then (\ref{lmt}) has a unique solution $\phi \in C_t\Sigma_0^1\cap L_t^{4}L_z^\infty \Sigma_x^1([-T_1,T_1]\times\R^3)$ such that $\partial_z\phi\in L_t^4L_z^\infty L_x^2([-T_1,T_1])$, depending continuously on $\psi_0$.
\end{thm}
\noindent
The proof is similar to Theorem \ref{wplmt4}.  Since $e^{it\partial_z^2/2}$ and $\partial_z$ are commutative, (\ref{lmt2-2}) and (\ref{lmt2-3}) are not necessary. Moreover for $\{ e^{it\partial_z^2/2} \}_{t\in\R}$ global Strichartz's estimate holds, so the restruction $T\le T_0(\le1)$ can be  eliminated. \\

Finally, we prove Theorem \ref{scatter}.
\begin{proof}[Proof of Theorem \ref{scatter}]
We define a scaling $J_\mu $ as $ (J_\mu \phi)(t,x,z) :=\mu^\f{1}{4} \phi(\mu^2 t, x, \mu z)$. Then 
\[ J_\mu^{-1} \Fav(J_\mu(\phi)(t)) = \Fav(\phi(t))\]
is an easy consequence.  By standard calculation, we obtain
\[ K[J_\mu\phi(t)]=\mu^{-\f{1}{2}}K[\phi(\mu^2t)] \qquad E[J_\mu \phi(t)]= \mu^{\f{3}{2}}E[\phi(\mu^2t)].  \] \\
\quad We prove a small data scattering. Take an initial data $\psi_0 \in \Sigma_0^1$ so that
\begin{equation}
\begin{split}
&\|e^{it\f{\partial_z^2}{2}} \psi_0\|_{L_t^4L_z^\infty \Sigma_x^1\cap L_t^{24}L_z^{\f{12}{5}} \Sigma_x^1 (\R)}^{\f{3}{4}}\|\partial_z e^{it\f{\partial_z^2}{2}} \psi_0\|_{L_t^{24}L_z^{\f{12}{5}} L_x^2(\R)}^{\f{1}{4}} \le C \|\psi_0\|_{ L_z^2\Sigma_x^1 }^{\f{3}{4}}  \| \partial_z \psi_0\|_{L_\bfx^2}^{\f{1}{4}} \le  \al,
\end{split}
\end{equation}
where $\al$ is the constant appeared in the statement of Theorem \ref{wplmt40}. From  Theorem \ref{wplmt40}, we have a global solution $\phi$ with  $\phi(0)=\psi_0$, and in the proof of the theorem we obtain a space time bound
\begin{equation*}
\begin{split}
\| e^{-i\theta\cH}\phi\|_{L_t^{\f{32}{3}} L_{\theta,\bfx}^{16}(\R \times[0,2\pi]\times \R^3)} 
&\lesssim \|\phi\|_{ L_t^4 L_z^\infty \Sigma_x^1(\R \times \R^3) }^{\f{1}{4}} \| \phi \|_{L_t^{24} L_z^{\f{12}{5}} \Sigma_x^1 (\R)}^{\f{1}{2}} \| \partial_z \phi\|_{L_t^{24} L_z^{\f{12}{5}} L_x^2(\R\times \R^3)}^{\f{1}{4}}\\
&\lesssim \|e^{it\f{\partial_z^2}{2}} \psi_0\|_{L_t^4L_z^\infty \Sigma_x^1\cap L_t^{24}L_z^{\f{12}{5}} \Sigma_x^1 (\R)}^{\f{3}{4}}\|\partial_z e^{it\f{\partial_z^2}{2}} \psi_0\|_{L_t^{24}L_z^{\f{12}{5}} L_x^2(\R)}^{\f{1}{4}}  < \infty.
\end{split}
\end{equation*}
Then,
\begin{equation*}
\begin{split}
 \Big\| \int_{t}^{+\infty} e^{i(t-s)\f{\partial_z^2}{2}}\Fav(\phi(s))ds \Big\|_{\Sigma_0^1}  \lesssim \|\phi\|_{L^\infty \Sigma_0^1(\R\times\R^3)}  \| e^{-i\theta\cH}\phi\|_{L_t^{\f{32}{3}} L_{\theta,\bfx}^{16}([t,+\infty] \times[0,2\pi]\times \R^3)}^8 \to 0\\
\end{split}
\end{equation*}
as $t \to +\infty$. Hence,
\[ \phi_+(t) := \psi_0-i\lmd \int_{0}^{+\infty} e^{i(t-s)\f{\partial_z^2}{2}}\Fav(\phi(s))ds \]
has the desired property. Similarly, we also have $\phi_{-} \in \Sigma_0^1$.
\end{proof}

\section{Strong magnetic confinement limit}\label{Sml}

\subsection{Proof of theorem\ref{main3}-(i)}

Theorem \ref{main3} follows from the next proposition.
\begin{prop}\label{3}
Fix $\sigma\in \{1,2 \}$. Suppose $\psiep_0, \phi_0 \in \Sigma^1$ satisfy
\begin{equation}\label{hyiv}
 \lim_{\ep \to +0 }\| \psiep_0 -\phi_0 \|_{\Sigma^1} = 0. 
\end{equation}
Let $\psiep \in C([0,\Tmax^\ep),\Sigma^1)$ be a maximal solution to (\ref{ep}) with the data $\psiep_0$ and $\phi \in C([0,\Tmax),\Sigma^1)$ be a maximal solution to (\ref{lmt}) with the data $\phi_0$. Then, for any small $\delta>0$, there exist $\ep_0=\ep_0(\phi_0, \delta)$, $T=T(\| \phi_0 \|_{\Sigma^1}) \in (0, \inf_{ \ep\in(0,\ep_0] }\{1,\Tmax^\ep, \Tmax \})$, and $C=C(\| \phi_0 \|_{\Sigma^1})>0$ such that

\begin{equation}\label{thm3}
\| e^{it\cH/\ep^2}\psiep-\phi \|_{L^\infty([0,T], \Sigma^1)}\le C(\delta +\| \psiep_0 - \phi_0 \|_{\Sigma^1})
\end{equation}
holds for all $\ep\in (0,\ep_0]$. Moreover, there exists $\ep_1=\ep_1(\| \phi_0 \|_{\Sigma^1})>0$ such that
\begin{equation}\label{thm3-2}
 \| e^{it\cH/\ep^2}\psiep-\phi \|_{L^\infty([0,T], L^2 )}\le C(\ep +\| \psiep_0 - \phi_0 \|_{L^2}) 
\end{equation}
holds for all $\ep \in (0,\ep_1]$.
\end{prop}

\begin{proof}
We first prove the case  $\sigma=1$. We consider the 3-dimensional harmonic oscillator
\[ \cH_{\bfx,0}=-\h\Delta_\bfx +\f{1}{8}|\bfx|^2. \]
We denote by $P_n (n\in\N_0)$ the  spectral projection  in  $L^2(\R^3)$ onto the eigenspace corresponding to eigenvalue $\f{2n+3}{4}$ (see Section \ref{Ap}). Then, the Hermite expansion of $u \in L^2(\R^3)$ is
\[ u(\bfx)=\sum_{n\ge0} P_n u(\bfx). \]
Let $P_{\le n}$ be the projection 
\[ P_{\le n} u (\bfx):= \sum_{m=0}^n P_m u(\bfx), \]
$P_{\le n}$ is bounded from $L^2$ to $\Sigma^k$ for any $k\in \N$.
We denote a solution to (\ref{lmt}) with an initial data $\phi_{0,n}:=P_{\le n} \phi_0$ as $\phi_n \in C([0,\Tmax^{(n)}), \Sigma^1)$ (where $\Tmax^{(n)}$ is the maximal time of existence) and denote $\tpsi_n(t):= e^{-it\cH/\ep^2} \phi_n(t)$, which  solves
\begin{equation*} 
i\partial_t\tpsi= (\f{\cH}{\ep^2}+\cH_z)\tpsi+\lmd F_{\text{av}}(\tpsi) \qquad \tpsi |_{t=0} = \phi_{0,n}.
\end{equation*}
Here, we use the following equality obtained by periodicity of $e^{i\theta\cH}$ (see (\ref{Fnl})):
\begin{equation*}
\begin{split}
F_{\text{av}}(e^{it\cH/\ep^2} u) &= \f{1}{2\pi}\int_0^{2\pi} e^{i\theta\cH}(|e^{-i\theta\cH+it\cH/\ep^2} u|^{2}e^{-i\theta\cH+it\cH/\ep^2} u)d\theta \\
&=\f{1}{2\pi}\int_{-t/\ep^2}^{2\pi-t/\ep^2} e^{i\theta\cH+it\cH/\ep^2}(|e^{-i\theta\cH} u|^{2}e^{-i\theta\cH} u)d\theta = e^{it\cH/\ep^2}F_{\mathrm{av}}(u) .
\end{split}
\end{equation*}
By Duhamel's formula, we have
\begin{equation}\label{lmt3}
\tpsi_n(t)= e^{-it(\cH/\ep^2+\cH_z)}\phi_{0,n}-i \lmd \int_0^t e^{-i(t-s)(\cH/\ep^2+\cH_z)}F_{\text{av}}(\tpsi_n(s))ds .
\end{equation}

Based on the assumption (\ref{hyiv}), there exists $\ep_0' >0$ such that for any $\ep \in (0,\ep_0']$,
\[ \| \psi^\ep_0-\phi_0 \|_{\Sigma^1} \le \f{1}{10} \| \phi_0 \|_{\Sigma^1},  \qquad  \text{i.e.}  \qquad   \f{9}{10} \|\phi_0\|_{\Sigma^1} \le \|\psiep_0\|_{\Sigma^1} \le \f{11}{10} \|\phi_0\|_{\Sigma^1}\]
and $n_0 \in \N$ such that for any $n_0\le n$, 
\[ \| \phi_{0,n}-\phi_0 \|_{\Sigma^1} \le \f{1}{10} \| \phi_0 \|_{\Sigma^1},  \qquad  \text{i.e.}  \qquad  \f{9}{10} \|\phi_0\|_{\Sigma^1} \le \|\phi_{0,n}\|_{\Sigma^1} \le \|\phi_0\|_{\Sigma^1}. \]
From the argument  of constructing a local solution to (\ref{ep}) and (\ref{lmt}),
\[ T_0:=\inf_{0<\ep \le \ep_0',\text{ }  n_0 \le n}\{ \Tmax^\ep, \Tmax^{(n)}\}\]
is positive. Therefore, in the following proof, we assume 

\begin{equation}\label{ass}
\ep \le \ep_0'  \qquad \text{and} \qquad  n\ge n_0.
\end{equation}

Let
\begin{equation}\label{dif}
\begin{split}
 u_n^\ep(t)&:=\psiep(t)-\tpsi_n(t) \\
 &= e^{-it(\cH/\ep^2+\cH_z)}(\psiep_0-\phi_{0,n}) \\
 &\quad -i\lmd \int_0^t e^{-i(t-s)(\cH/\ep^2+\cH_z)}\big(|\psiep(s)|^{2}\psiep(s)-|\tpsi_n(s)|^{2}\tpsi_n(s)\big)ds \\
 & \quad -i\lmd \int_0^t e^{-i(t-s)(\cH/\ep^2+\cH_z)}\big(|\tpsi_n(s)|^{2}\tpsi_n(s)-F_{\text{av}}(\tpsi_n(s))\big)ds \\
 &=:e^{-i(t-s)(\cH/\ep^2+\cH_z)}(\psiep_0-\phi_{0,n})-i\lmd(A_1+A_2).
\end{split}
\end{equation}

For $T\in(0, \min\{T_0,1\}]$, we have
\begin{equation}\label{estu1}
 \| u^{\ep}_n \|_{L_t^\infty( [0,T],\Sigma^1)}\le C\| (\psiep_0-\phi_{0,n}) \|_{\Sigma^1} +\|A_1 \|_{L_t^\infty ([0,T],\Sigma^1)}+\|A_2 \|_{L_t^\infty ([0,T], \Sigma^1)} .
 \end{equation}

\quad\\
\noindent
\underline{Estimate of $A_1$}\\

Similarly to the bound of $\Psi^\ep$ in the proof of Theorem \ref{wpep1}, we have
\begin{equation*}
\begin{split}
\|A_1 \|_{L_t^\infty \Sigma^1}&\lesssim  \| u_n^\ep\|_{L_t^\infty\Sigma^1\cap L_t^4\Sigma_{x,z}^{4,2}}( \| \psiep \|_{L_t^\infty\Sigma^1\cap L_t^4\Sigma_{x,z}^{4,2}}^2 +\| \tpsi_n \|_{L_t^\infty\Sigma^1\cap L_t^4\Sigma_{x,z}^{4,2}}^2 )T^{\f{1}{4}}.
\end{split}
\end{equation*}
By the definition of $M(a,T)$ in the proof of Theorem \ref{wpep1}, (\ref{T1}), and (\ref{ass}), for any small $T$ such that
\begin{equation}\label{T10}
T\le \f{1}{\tilde{C}} \min\{ \f{1}{\|\phi_0 \|_{\Sigma^1}^{8}}, 1\} \lesssim \min \{\f{1}{\|\psiep_0 \|_{\Sigma^1}^{8}}, 1\} ,
\end{equation}
(where $\tilde{C}\ge1$ is a sufficiently large constant) we have
\[ \| \psiep \|_{L_t^\infty \Sigma^1\cap L_t^4\Sigma_{x,z}^{4,2}([0,T])} \lesssim \| \psiep_0 \|_{\Sigma^1} \lesssim  \| \phi_0 \|_{\Sigma^1}. \]
Taking the $L_t^\infty \Sigma^1\cap L_t^4\Sigma_{x,z}^{4,2}([0,T])$ norm on (\ref{lmt3}), we have
\begin{equation*}
\begin{split}
\| \tpsi_n\|_{L_t^\infty \Sigma^1 \cap L_t^4\Sigma_{x,z}^{4,2}}&\lesssim \|  \phi_{0,n}\|_{\Sigma^1}+ (\| \nabla_\bfx \Fav(\tpsi_n(t)) \|_{  L_t^1 L_x^2L_z^2}+\| \langle \bfx \rangle \Fav(\tpsi_n(t)) \|_{  L_t^1 L_x^2L_z^2}\\
&\lesssim \|  \phi_{0,n}\|_{\Sigma^1}+ \| \tpsi_n \|_{L_t^\infty \Sigma^1}  \| \tpsi_n \|_{L_t^\infty L^2} \| \partial_z \tpsi_n \|_{L_t^\infty L^2} T\\
&\lesssim \|  \phi_{0,n}\|_{\Sigma^1}+ \| \phi_n \|_{L_t^\infty \Sigma^1}^3 T .
\end{split}
\end{equation*}
By the definition of $M(a,T)$ in the proof of Theorem \ref{wplmt1}, (\ref{T0}), and (\ref{ass}), for any small $T$ such that
\begin{equation}\label{T4}
T \lesssim \min\{ \f{1}{\|\phi_{0}\|_{ \Sigma^1}^2}, 1\}  \lesssim \min\{ \f{1}{\|\phi_{0,n}\|_{ \Sigma^1}^2}, 1\} ,
\end{equation}
we have
\begin{equation*}
\| \phi_n\|_{L^\infty([0,T], \Sigma^1)} \lesssim \|\phi_{0,n}\|_{\Sigma^1} \lesssim \|  \phi_{0}\|_{\Sigma^1} . 
\end{equation*}
Hence,
\begin{equation*}
\| \tpsi_n\|_{L_t^\infty\Sigma^1\cap L_t^4\Sigma_{x,z}^{4,2}([0,T])} \lesssim \|  \phi_{0}\|_{\Sigma^1} . 
\end{equation*}

To summarize the argument presented so far, if  $T$ satisfies (\ref{T10}) and (\ref{T4}), and the constant $\tilde{C}$ in (\ref{T10}) is sufficiently large,  $A_1$ is bounded as
\begin{equation*}
\|A_1 \|_{L_t^\infty \Sigma^1([0,T])}\le C \| u_n^\ep\|_{L_t^\infty \Sigma^1\cap L_t^4\Sigma_{x,z}^{4,2}([0,T])}\|\phi_0\|_{\Sigma^1}^2T^{\f{1}{4}}\le \f{1}{10} \| u_n^\ep\|_{L_t^\infty \Sigma^1\cap L_t^4\Sigma_{x,z}^{4,2}([0,T])}.
\end{equation*}
Then, we have
\begin{equation*}
\begin{split}
\| u_n^\ep \|_{L_t^\infty \Sigma^1([0,T]) }&\le C\| \psiep_0-\phi_{0,n} \|_{\Sigma^1} +(\|A_1 \|_{L_t^\infty \Sigma^1([0,T])}+\|A_2 \|_{L_t^\infty \Sigma^1([0,T])} ) \\
&\le C\| \psiep_0-\phi_{0,n} \|_{\Sigma^1} +\f{1}{10} \| u_n^\ep\|_{L_t^\infty\Sigma^1 \cap L_t^4 \Sigma_{x,z}^{4,2}}+\|A_2 \|_{L_t^\infty \Sigma^1} .
\end{split}
\end{equation*}\\
 Taking the $L_t^4\Sigma_{x,z}^{4,2}([0,T])$ norm on (\ref{dif}) and iterating the above argument, we have 
 \begin{equation*}
\begin{split}
\| u_n^\ep \|_{L_t^4\Sigma_{x,z}^{4,2}([0,T])} &\le C\| \psiep_0-\phi_{0,n} \|_{\Sigma^1} +\f{1}{10} \| u_n^\ep\|_{L_t^\infty\Sigma^1 \cap L_t^4 \Sigma_{x,z}^{4,2}}+\|A_2 \|_{L_t^4 \Sigma_{x,z}^{4,2}}.
\end{split}
\end{equation*}
Combining the two inequalities presented above, we obtain
\begin{equation}\label{uep2}
\begin{split}
\| u_n^\ep \|_{L_t^\infty \Sigma^1\cap L_t^4\Sigma_{x,z}^{4,2}([0,T])} \lesssim \| \psiep_0-\phi_{0,n} \|_{\Sigma^1} +\|A_2 \|_{L_t^\infty \Sigma^1\cap L_t^4\Sigma_{x,z}^{4,2}([0,T])}.
\end{split}
\end{equation}\\

\noindent
\underline{Estimate of $A_2$}\\

We introduce the following function, defined on $\R\times \Sigma^2$,
\[  \cF(s, u):=\int_0^s (F(\theta,u)-F_{\text{av}}(u))d\theta .\]
Because $F(\cdot, u)$ is $2\pi$-periodic and $F_{\text{av}}(u)$ is the time averaged function of  $F(\cdot, u)$, it follows that
\[ \int_0^s (F(\theta,u)-F_{\text{av}}(u))d\theta = \int_{2\pi\lfloor s/2\pi \rfloor}^s (F(\theta,u)-F_{\text{av}}(u))d\theta, \] 
where $\lfloor\cdot \rfloor$ is the floor function. Here, we calculate
\begin{equation}\label{ibp}
\begin{split}
&e^{-i(t-s)\cH_z} \Big(F(\f{s}{\ep^2}, \phi_n(s))-F_{\text{av}}(\phi_n(s))\Big)\\
&=\ep^2\f{d}{ds}\Big(e^{-i(t-s)\cH_z}\cF(\f{s}{\ep^2}, \phi_n(s))\Big)-i\ep^2e^{-i(t-s)\cH_z}\cH_z\cF(\f{s}{\ep^2}, \phi_n(s))\\
&\qquad -\ep^2 e^{-i(t-s)\cH_z}D_u\cF(\f{s}{\ep^2}, \phi_n(s))[ \partial_t \phi_n(s)]
\end{split}
\end{equation}
where $D_u\cF(s, u)[v]$ is
\begin{equation*}
\begin{split}
D_u\cF(s, u)[v]=& \int_0^s (F(\theta,v,u,u)-F_{\text{av}}(v,u,u))d\theta \\
&+ \int_0^s (F(\theta,u,v,u)-F_{\text{av}}(u,v,u))d\theta + \int_0^s (F(\theta,u,u,v)-F_{\text{av}}(u,u,v))d\theta \\
\end{split}
\end{equation*}
with
\[ F(\theta, u_1,u_2,u_3):= e^{i\theta\cH}((e^{-i\theta\cH}u_1)\overline{(e^{-i\theta\cH}u_2)}(e^{-i\theta\cH}u_3)) .\]
We begin from the estimate of $A_2$. By the periodicity of $\Fav$, 
\begin{equation*}
 A_2= e^{-it\cH/\ep^2} \int_0^t e^{-i(t-s)\cH_z}\big(F(\f{s}{\ep^2},\phi_n(s))-F_{\text{av}}(\phi_n(s))\big)ds.
\end{equation*}
Applying (\ref{ibp}) to $A_2$, we obtain
\begin{equation}\label{A2}
\begin{split}
\| A_2\|_{L_t^\infty \Sigma^1} &\lesssim \ep^2 \Big\| \cF(\f{t}{\ep^2}, \phi_n(t)) \Big\|_{L_t^\infty \Sigma^1}+\ep^2 \Big\|  e^{-it\cH/\ep^2}\int_0^t  e^{is\cH_z} \cH_z\cF(\f{s}{\ep^2}, \phi_n(s)) ds \Big\|_{L_t^\infty \Sigma^1}\\
& \qquad + \ep^2\Big\|  e^{-it\cH/\ep^2}\int_0^t e^{is\cH_z}D_u\cF(\f{s}{\ep^2}, \phi_n(s))[ \partial_t\phi_n(s) ]ds \Big\|_{L_t^\infty \Sigma^1}\\
&=:\ep^2(I_1+I_2+I_3).
\end{split}
\end{equation}

\noindent
We estimate $I_1$. First, we have
\begin{equation}\label{calF}
 \Big\| \cF(\f{t}{\ep^2}, \phi_n(t)) \Big\|_{L_t^\infty \Sigma^1} \le \Big\| \int_{2\pi \lfloor t/2\pi\ep^2\rfloor}^{t/\ep^2} F(\theta, \phi_n(t))d\theta \Big\|_{L_t^\infty \Sigma^1} + \Big\| \int_{2\pi \lfloor t/2\pi\ep^2\rfloor}^{t/\ep^2} F_{\text{av}}(\phi_n(t))d\theta \Big\|_{L_t^\infty \Sigma^1} .
 \end{equation}
Because $|t/\ep^2-2\pi \lfloor t/2\pi\ep^2\rfloor| \le 2\pi$ holds for any $\ep$ and $t$,  we estimate the right-hand side of (\ref{calF}) as (\ref{Favbdd1}) to obtain
 \begin{equation}\label{estF}
\begin{split}
\Big\| \int_{2\pi \lfloor t/2\pi\ep^2\rfloor}^{t/\ep^2}  F(\theta, \phi_n(t)) d\theta \Big\|_{L_t^\infty \Sigma^1}
 \lesssim  \| \phi_n \|_{L_t^\infty \Sigma^1}^3
\end{split}
\end{equation}
and
\[ \Big\| \int_{2\pi \lfloor t/2\pi\ep^2\rfloor}^{t/\ep^2} F_{\text{av}}(\phi_n(t))d\theta \Big\|_{L_t^\infty \Sigma^1}  \lesssim  \| \phi_n \|_{L_t^\infty \Sigma^1}^3. \]
Hence, we have
\begin{equation*}
\Big\| \cF(\f{t}{\ep^2}, \phi_n(t)) \Big\|_{L_t^\infty \Sigma^1}  \lesssim  \| \phi_n \|_{ L_t^\infty\Sigma^1}^3.
\end{equation*}\\

\noindent
Next, we address $I_2$. We decompose
\begin{equation*}
\begin{split}
 &\Big\|  e^{-it\cH/\ep^2}\int_0^t  e^{is\cH_z} \cH_z\cF(\f{s}{\ep^2}, \phi_n(s)) ds \Big\|_{L_t^\infty \Sigma^1} \\
 \le&  \Big\|  e^{-it\cH/\ep^2}\int_0^t  e^{is\cH_z} \cH_z \int_{2\pi \lfloor s/2\pi\ep^2\rfloor}^{s/\ep^2} F(\theta, \phi_n(s)) d\theta ds \Big\|_{L_t^\infty \Sigma^1} \\
 &\quad +\Big\|  e^{-it\cH/\ep^2}\int_0^t  e^{is\cH_z} \cH_z \int_{2\pi \lfloor s/2\pi\ep^2\rfloor}^{s/\ep^2} F_{\text{av}}(\phi_n(s)) d\theta ds \Big\|_{L_t^\infty \Sigma^1}\\
 =:& J_1+J_2.
\end{split}
\end{equation*}
We now estimate $J_1$.  By a change of variable  $\theta= \tilde{\theta}+s/\ep^2$, we have
\begin{equation*}
\begin{split}
J_1 &= \Big\| e^{-it\cH/\ep^2}\int_0^t e^{is\cH_z} \cH_z \int_{-s/\ep^2+2\pi \lfloor s/2\pi\ep^2\rfloor}^{0}  e^{is\cH/\ep^2}  F(\tilde{\theta}, \tpsi_n(s)) d\tilde{\theta} ds \Big\|_{L_t^\infty \Sigma^1} \\
&=  \Big\|\int_0^t e^{-i(t-s)\cH/\ep^2} G_\ep(s,\bfx) ds\Big\|_{L_t^\infty \Sigma^1} 
\end{split}
\end{equation*}
where
\[ G_\ep(s,\bfx):=  e^{is\cH_z} \cH_z \int_{-s/\ep^2+2\pi \lfloor s/2\pi\ep^2\rfloor}^{0}  F(\tilde{\theta}, \tpsi_n(s))  d\tilde{\theta}. \]
 By  (\ref{eqH}), Minkowski's inequality, and Strichartz's estimate, we have
\begin{equation*}
\begin{split}
&\Big\|\int_0^t e^{-i(t-s)\cH/\ep^2} G_\ep(s,\bfx) ds\Big\|_{L_t^\infty \Sigma^1} \\
\lesssim& \Big\|\int_0^t \cH_0^\h e^{-i(t-s)\cH/\ep^2} G_\ep(s,\bfx) ds\Big\|_{L_z^2 L_t^\infty L_x^2} + \Big\|\int_0^t \partial_z e^{-i(t-s)\cH/\ep^2} G_\ep(s,\bfx) ds\Big\|_{L_z^2 L_t^\infty L_x^2}\\
& +\Big\|\int_0^t z e^{-i(t-s)\cH/\ep^2} G_\ep(s,\bfx) ds\Big\|_{L_z^2 L_t^\infty L_x^2} \\
\lesssim& \| \cH_0^\h G_\ep(t,\bfx) \|_{L_x^2 L_t^1L_x^2}+ \| \partial_z G_\ep(t,\bfx) \|_{L_z^2 L_t^1 L_x^2} +\| z G_\ep(t,\bfx) \|_{L_z^2 L_t^1 L_x^2} \\
\lesssim& \| \cH_0^\h G_\ep(t,\bfx) \|_{ L_t^1L_\bfx^2} +  \| G_\ep(t,\bfx) \|_{ L_t^1 L_x^2 \Sigma_z^1}.
\end{split}
\end{equation*}
For the first term, by Lemma \ref{sigmaz2}, we have
\begin{equation*}
\begin{split}
  \| \cH_0^\h G_\ep(t,\bfx) \|_{ L_t^1L_\bfx^2} 
  &= \Big\|  \cH_z \cH_0^\h\int_{ -t/\ep^2 + 2\pi \lfloor t/2\pi\ep^2 \rfloor}^{0} F(\tilde{\theta}, \tpsi_n)   d\tilde{\theta}  \Big\|_{ L_t^1 L_\bfx^2 } \\
 &\lesssim \Big\|  \cH_0^\h\int_{ -t/\ep^2 + 2\pi \lfloor t/2\pi\ep^2 \rfloor}^{0} F(\tilde{\theta}, \tpsi_n)   d\tilde{\theta}  \Big\|_{L_t^1 L_x^2 \Sigma_z^2}\\
 &\lesssim \Big\|  \int_{2\pi \lfloor t/2\pi\ep^2\rfloor}^{t/\ep^2}  F(\theta, \phi_n) d\theta \Big\|_{L_t^1 \Sigma^3}.\\
\end{split}
\end{equation*}
For the second term, we apply Lemma \ref{sigmaz} and Lemma \ref{sigmaz2}, (Because we assume $T\le1$, the constant can be independent of $T$) 
\begin{equation*}
\begin{split}
  \| G_\ep(t,\bfx) \|_{ L_t^1 L_x^2 \Sigma_z^1} &\lesssim \Big\|  \int_{ -t/\ep^2 + 2\pi \lfloor t/2\pi\ep^2 \rfloor}^{0} F(\tilde{\theta}, \tpsi_n)   d\tilde{\theta}  \Big\|_{L_t^1 L_x^2 \Sigma_z^3}\\
 &\lesssim \Big\|  \int_{2\pi \lfloor t/2\pi\ep^2\rfloor}^{t/\ep^2}  F(\theta, \phi_n) d\theta \Big\|_{L_t^1 \Sigma^3}.\\
\end{split}
\end{equation*}
Similarly to (\ref{estF}), we have
\begin{equation*}
\begin{split}
J_1\lesssim \Big\| \int_{2\pi \lfloor t/2\pi\ep^2\rfloor}^{t/\ep^2}  F(\theta, \phi_n) d\theta \Big\|_{L_t^1 \Sigma^3}
 \lesssim \| \phi_n \|_{ L_t^\infty \Sigma^3} \| \phi_n \|_{L_t^\infty \Sigma^1}^2.
\end{split}
\end{equation*}
We also have the same bound for $J_2$.\\

\noindent
 $I_3$ is treated similarly to $I_2$ using the following bound
\[ \| \partial_t \phi_n \|_{\Sigma^1}\lesssim \| \phi_n \|_{\Sigma^3} +\|\phi_n \|_{\Sigma^1}^2\| \phi_n \|_{L_\bfx^2}, \]
which is obtained by taking the $\Sigma^1 $ norm for the both sides of (\ref{lmt}).\\

\noindent
From these estimates, we have
\begin{equation}\label{A22}
\|A_2\|_{L_t^\infty \Sigma^1}\lesssim \ep^2 ( \| \phi_n \|_{ L_t^\infty \Sigma^3} +\| \phi_n \|_{L_t^\infty\Sigma^1}^3 )\| \phi_n \|_{L_t^\infty \Sigma^1}^2.
\end{equation}

\noindent
Finally, for $\|A_2\|_{L_{t}^4 \Sigma_{x,z}^{4,2}}$ we have
\begin{equation*}
\begin{split}
\| A_2\|_{L_{t}^4 \Sigma_{x,z}^{4,2}} &\le \ep^2 \Big\| e^{-it\cH/\ep^2} \cF(\f{t}{\ep^2}, \phi_n(t)) \Big\|_{L_{t}^4 \Sigma_{x,z}^{4,2}}\\
& \quad +\ep^2 \Big\| e^{-it\cH/\ep^2} \int_0^t  e^{is\cH_z} \cH_z\cF(\f{s}{\ep^2}, \phi_n(s)) ds \Big\|_{L_{t}^4 \Sigma_{x,z}^{4,2}}\\
& \quad + \ep^2 \Big\| e^{-it\cH/\ep^2} \int_0^t e^{is\cH_z}D_u\cF(\f{s}{\ep^2}, \phi_n(s))[\partial_t\phi_n(s)] ds \Big\|_{L_{t}^4 \Sigma_{x,z}^{4,2}}.
\end{split}
\end{equation*}
The second and third terms on the right-hand side can be treated similar to  $I_2$ and $I_3$, respectively.
For fhe first term, we use the embedding $ \Sigma^2  \hookrightarrow \Sigma_{x,z}^{4,2}$. Then, we have
\begin{equation}\label{A23}
\begin{split}
\|A_2\|_{L_{t}^4 \Sigma_{x,z}^{4,2}([0,T])} \lesssim \ep^2 (\| \phi_n \|_{ L_t^\infty \Sigma^3}+\| \phi_n \|_{L_t^\infty\Sigma^1}^3 ) \| \phi_n \|_{L_t^\infty \Sigma^1}^2.
\end{split}
\end{equation}
We obtain the bound of $A_2$.\\

By (\ref{perlmt1}) and continuous dependence,  if we take $n$ sufficiently large, for any positive integer $k$ and any small $T \lesssim  \min\{ \|  \phi_0\|_{\Sigma^1}^{-2}, 1\}$, we have
\[  \| \phi_n \|_{ L_t^\infty \Sigma^k([0,T])} \lesssim  \| \phi_{0,n} \|_{\Sigma^k} \]
and
\[ \| \phi-\phi_n \|_{L_t^\infty \Sigma^1([0,T])} \lesssim  \| \phi_0-\phi_{0,n} \|_{\Sigma^1} . \]
Hence, we have
\begin{equation}\label{uep3}
\begin{split}
\| u_n^\ep \|_{L_t^\infty \Sigma^1} \lesssim \| \psiep_0-\phi_{0,n} \|_{\Sigma^1} +\ep^2 (\| \phi_{0,n} \|_{\Sigma^3}  + \| \phi_{0,n} \|_{\Sigma^1}^3 ) \| \phi_{0,n} \|_{\Sigma^1}^2.
\end{split}
\end{equation}
From the argument presented so far, if we set $\ep$ small and $n$ large depending on $\delta$ and $\phi_0$, and set $T$ small depending on $\| \phi_0 \|_{\Sigma^1}$, we conclude
\begin{equation*}
\begin{split}
\| e^{it\cH/\ep^2}\psiep-\phi &\|_{L^\infty([0,T], \Sigma^1)} \lesssim \| \phi-\phi_n \|_{L^\infty ([0,T], \Sigma^1)}+\| u_n^\ep \|_{L^\infty([0,T], \Sigma^1)}\\
&\lesssim \| \phi_0-\phi_{0,n} \|_{\Sigma^1} + \| \psiep_0-\phi_0 \|_{\Sigma^1} +\ep^2( \| \phi_{0,n} \|_{\Sigma^3}  + \| \phi_{0,n} \|_{\Sigma^1}^3) \| \phi_{0,n} \|_{\Sigma^1}^2\\
&\lesssim \delta +\ep^2 n^2\| \phi_0 \|_{\Sigma^1}^3+ \ep^2\| \phi_0 \|_{\Sigma^1}^5 + \| \psiep_0-\phi_0 \|_{\Sigma^1}  \\ 
&\le \delta(1+\| \phi_0 \|_{\Sigma^1}^3+ \| \phi_0 \|_{\Sigma^1}^5) +  \| \psiep_0-\phi_0 \|_{\Sigma^1}.\\ 
\end{split}
\end{equation*}
\quad To prove (\ref{main3-2}), we conduct the same argument for the $ L_t^\infty L_{\bfx}^2 $ norm instead of the $ L_t^\infty \Sigma^1 $ norm.  Then, we have 
\begin{equation*}
\begin{split}
&\|e^{it\cH/\ep^2}\psiep-\phi \|_{L^\infty ([0,T],L_\bfx^2)} \\
\lesssim& \| \phi_0-\phi_{0,n} \|_{L_\bfx^2} + \| \psiep_0-\phi_0 \|_{L_\bfx^2} +\ep^2 (\| \phi_{0,n} \|_{\Sigma^2} + \| \phi_{0,n} \|_{\Sigma^1}^3) \| \phi_{0,n} \|_{\Sigma^1}^2\\
 \lesssim&  \f{1}{n} \| \phi_0 \|_{\Sigma^1}+\ep^2 n \| \phi_0 \|_{\Sigma^1}^3 + \ep^2\| \phi_0 \|_{\Sigma^1}^5  + \| \psiep_0-\phi_0 \|_{L_\bfx^2}.
\end{split}
\end{equation*}
If we set $\ep$ small and $n$ large depending on $\| \phi_0 \|_{\Sigma^1}$, maintaining the relation $\ep \simeq n^{-1}$, we conclude
\[ \| e^{it\cH/\ep^2}\psiep-\phi \|_{L^\infty([0,T], L_\bfx^2)} \lesssim \ep(1+\| \phi_0 \|_{\Sigma^1}^3+ \| \phi_0 \|_{\Sigma^1}^5)+\| \psiep_0-\phi_0 \|_{L_\bfx^2} .\]
Proof for the case $\sigma=1$ is complete. \\

Next, we prove the case $\sigma=2$. This proof follows the same strategy as  that adopted for the case $\sigma=1$. However, we need two modification because when $\sigma=2$, (\ref{ep}) is energy-critical. First, by (\ref{hyiv}), (\ref{assep}), and (\ref{Tec2}), there exists $\ep_0'>0$ depending on $\phi_0$ such that 
\[ \|\psi_0^\ep -\phi_0\|_{\Sigma^1} \le \f{1}{10}\|\phi_0\|_{\Sigma^1} \text{\qquad and \qquad} \inf\{ \Tmax^\ep | \ep \in (0,\ep_0'] \}>0. \]  
Second, we cannot obtain a positive power $T$ in the bound of $A_1$ (see (\ref{estu1})).
We improve this result based on the constant of Strichartz's estimate, that is, we bound
\begin{equation*}
\begin{split}
\|A_1 \|_{L_t^\infty \Sigma^1([0,T])}&\lesssim  \| u_n^\ep\|_{L_t^\infty\Sigma^1\cap L_t^3\Sigma_{x,z}^{6,2}}( \| \psiep \|_{L_t^\infty\Sigma^1 \cap L_t^3\Sigma_{x,z}^{6,2}}^4 +\| \tpsi_n \|_{L_t^\infty\Sigma^1 \cap L_t^3\Sigma_{x,z}^{6,2}}^4 )(\ep^2+T)^{\f{1}{3}}.\\
\end{split}
\end{equation*}\\
Based on the definition of $M(a,b,T)$ in the proof of Theorem \ref{wpep2}, (\ref{assep}), and (\ref{hyiv}), if we take $\ep \in (0,\ep_0']$ sufficiently small depending on $\|\phi_0\|_{\Sigma^1}$, we have
\[ \|\psiep\|_{L_t^\infty \Sigma^1\cap L_t^3\Sigma_{x,z}^{6,2}([0,T])}\lesssim \|\psiep_0\|_{\Sigma^1} \lesssim \|\phi_0\|_{\Sigma^1} \]
for any small $T\lesssim \min\{ \|\phi_0\|_{\Sigma^1}^{-12},1 \}$. For the estimate of $\tpsi$, we also take the $L_t^\infty \Sigma^1\cap L_t^3\Sigma_{x,z}^{6,2}([0,T])$ norm on (\ref{lmt3}) and apply Stricrartz's estimate and (\ref{Favbdd2}). Accordingly, we have
\begin{equation}
\begin{split}
\|\tpsi_n\|_{L_t^\infty\Sigma^1 \cap L_t^3\Sigma_{x,z}^{6,2}([0,T])} &\lesssim \|\phi_{0,n}\|_{\Sigma^1}+ \|\Fav(\tpsi_n)\|_{L_t^1\Sigma^1([0,T])} \\
&\lesssim \|\phi_{0,n}\|_{\Sigma^1} + \|\tpsi_n\|_{L_t^\infty \Sigma^1}^5T \\
&\lesssim \|\phi_{0,n}\|_{\Sigma^1} + \|\phi_n\|_{L_t^\infty \Sigma^1}^5T.
\end{split}
\end{equation}
If we choose a small $T\lesssim \min\{ \|\phi_0\|_{\Sigma^1}^{-4},1\}$ and sufficiently large $n$, we have

\begin{equation}
\|\tpsi_n\|_{L_t^\infty\Sigma^1 \cap L_t^3\Sigma_{x,z}^{6,2}([0,T])} \lesssim \|\phi_0\|_{\Sigma^1}.
\end{equation}
Thus, choosing $\ep^2\le T\lesssim \min\{ \|\phi_0\|_{\Sigma^1}^{-12}, \|\phi_0\|_{\Sigma^1}^{-4}, 1\} $, we obtain  the bound of $A_1$.\\

$A_2$ can be treated as that depicted for the case $\sigma=1$ by replacing the $L_t^4\Sigma_{x,z}^{4,2}$ norm with the $L_t^3\Sigma_{x,z}^{6,2}$ norm.
\end{proof}

\begin{proof}[Proof of Theorem\ref{main3}-(i)]
Set $T\in (0, \Tmax)$ arbitrarily. Because $\| \phi \|_{L^\infty([0,T], \Sigma^1)} < \infty$ from Theorem \ref{main2}, we can iterate Proposition \ref{3} until the life span reaches $T$.
\end{proof}

\subsection{Proof of theorem\ref{main3}-(ii)}

First, we briefly describe the local well-posedness of (\ref{ep}) in $\Sigma^2$ for all $\sigma \in \N$. We use the function space
\[ \{ \psi \in C([0,T],\Sigma^2)  :  \|\psi\|_{L^\infty([0,T],\Sigma^2)}\le a \} \]
To construct the local solution, we use (\ref{eqH}), Lemma \ref{sigmaz}, and the fact $\Sigma^2$ is a Banach algebra. Then, we have
\begin{equation}
\begin{split}
\| \Psi^\ep[\psi]\|_{L_t^\infty\Sigma^2([0,T])} &\le C_0 \|\psi_0\|_{\Sigma^2}+C\| \psi\|_{L_t^\infty\Sigma^2([0,T])}^{2\sigma+1}T \\
&\le C_0 \|\psi_0\|_{\Sigma^2}+Ca^{2\sigma+1}T.
\end{split}
\end{equation}
Choose $a=2C_0\|\psi_0\|_{\Sigma^2}$ and $T\le \min\{C^{-1}a^{-2\sigma},1 \} $.\\

Similarly to the previous subsection, the following proposition is needed to prove Theorem \ref{main3}-(ii).

\begin{prop}\label{4}
Fix $\sigma\in \N$. Suppose $\psiep_0, \phi_0 \in \Sigma^2$ satisfy
\begin{equation}\label{hyiv2}
\lim_{\ep \to +0 }\| \psiep_0 -\phi_0 \|_{\Sigma^2} = 0.
\end{equation}
Let $\psiep \in C([0,\Tmax^\ep),\Sigma^2)$ be a maximal solution to (\ref{ep}) with the data $\psiep_0$, and $\phi \in C([0, \Tmax), \Sigma^2)$ be a maximal solution to (\ref{lmt}) with the data $\phi_0$. Then, for any $\delta>0$, there exist $\ep_0=\ep_0(\phi_0, \delta)$,  $T=T(\| \phi_0 \|_{\Sigma^2})\in (0,\inf_{\ep \in (0,\ep_0]}\{1,\Tmax^\ep,\Tmax \})$, and $C=C(\| \phi_0 \|_{\Sigma^2})>0$ such that
\begin{equation}\label{thm3}
\| e^{it\cH/\ep^2}\psiep-\phi \|_{L^\infty([0,T], \Sigma^2)}\le C(\delta +\| \psiep_0 - \phi_0 \|_{\Sigma^2})
\end{equation}
holds for all $\ep\in (0,\ep_0]$.
\end{prop}

\begin{proof}
We only prove the case $\sigma=3$. Define $\phi_n, \tpsi_n, u_n, A_1,A_2$, etc.~as those defined in the proof of Proposition \ref{3}. 
Taking $\ep$ sufficiently small, $n$ sufficiently large, and $T$ sufficiently small as $T\lesssim \min\{  \|\phi_0 \|_{\Sigma^2}^{-6},1\}$,  we have
\begin{equation}
\begin{split}
A_1 &\le C \|u_n\|_{\Sigma^2}(\|\psiep\|_{L_t^\infty\Sigma^2}^6+\|\tpsi_n\|_{L_t^\infty\Sigma^2}^6)T\\
&\le C \|u_n\|_{\Sigma^2}(\|\psiep_0\|_{L_t^\infty\Sigma^2}^6+\|\phi_{0,n}\|_{L_t^\infty\Sigma^2}^6)T\\
&\le \f{1}{10} \|u_n\|_{\Sigma^2},
\end{split}
\end{equation}
where we use the Moser-type inequality,
\[ \||f|^{2\sigma}f\|_{\Sigma^2} \lesssim \|f\|_{L^\infty}^{2\sigma}\|f\|_{\Sigma^2}. \]
Thus, 
\[ \|u_n\|_{L_t^\infty\Sigma^2}\lesssim \|\psi_0-\phi_{0,n}\|_{\Sigma^2}+ \|A_2\|_{L_t^\infty\Sigma^2} \]
Estimate of $A_2$ is similar to Proposition \ref{3}. Because $ \Sigma^2 $ is a Banach algebra, we do not use the $L^{12}([0,T], \Sigma_{z,x}^{3,2})$ norm.
\end{proof}

\begin{proof}[Proof of theorem\ref{main3}-(ii)]
Set $T\in (0,\Tmax)$ arbitrarily. Because $\| \phi \|_{L^\infty([-T,T], \Sigma^2)} < \infty$ holds from Theorem \ref{main2}, we can iterate Proposition \ref{4} until the life span reaches $T$.
\end{proof}

\section{Appendix}\label{Ap}

Here, we introduce some properties of operators appearing in this study. For details, refer to \cite{H}. See also \cite{CR} for the 2D case. First, we define a linear operator on $L^2(\R^d)$
\[ H:=-\Delta + \om^2|x|^2  \qquad \om>0.\]
This is harmonic oscillator.\\

When we denote $x=(x_1, \cdots,x_d)$, the Hermite function $h_{\mathbf{k}}$ $(\mathbf{k}=(k_1,\cdots,k_d) \in \N_0^d)$ is defined as follows. When $d=1$, let
\[ H_k(x):=(-1)^ke^{\om x^2} \f{d^k}{dx^k}(e^{-\om x^2}), \]
then, $h_k$ is defined as
\begin{equation}
 h_k(x):=\f{ e^{- \om x^2/2}H_k(x)}{(\om^{k-1/2}\pi^{1/2}2^k k!)^{1/2}}. 
 \end{equation}
When $d\ge2$, 
\begin{equation} h_{\mathbf{k}}(x):=\Pi_{j=1}^d h_{k_j}(x_j). \end{equation}
Then, the Hermite function satisfies 
\[ H h_{\mathbf{k}}=(2|\mathbf{k}|+d)\om  h_{\mathbf{k}}  \qquad |\mathbf{k}|=k_1+\cdots+k_d  \]
and $\{ h_{\mathbf{k}} \}_{\mathbf{k} \in \N_0^d}$ is the Hilbertian bases of $L^2(\R^d)$.\\

\noindent
Particularly, when $d=2$ and $\om=\h$, 
\[ \cH_0:= -\f{1}{2}\Delta_x+\f{1}{8}|x|^2 \qquad  \text{and} \qquad L:=-\f{i}{2}x^\perp\cdot\nabla_x  \]
satisfy
\[ \cH_0 h_{\mathbf{k}}=\h(k_1+k_2+1)h_{\textbf{k}} \qquad \text{and} \qquad  L h_{\textbf{k}}=\h(k_1-k_2)h_{\mathbf{k}} \]
respectively. (Hence, $E^\ep[\psi]$ cannot control $\|\psi\|_{L_z^2\Sigma_x^1}$.)

\section*{Acknowledgments}
 The author is deeply grateful to Kenji Nakanishi, for useful his advice for this study.
 
 The author is supported by JST, the establishment of university fellowships towards the creation of science technology innovation, Grant Number JPMJFS2123.
 
 The author would like to thank Editage (www.editage.jp) for English language editing.

\bibliographystyle{abbrv}
\footnotesize{
\bibliography{NLSavbib}

\begin{thebibliography}{10}

\bibitem{N.Ben2}
N.~Ben~Abdallah, F.~Castella, F.~Delebecque-Fendt, and F.~Méhats.
\newblock The strongly confined {S}chrödinger-poisson system for the transport
  of electrons in a nanowire.
\newblock {\em SIAM J. Appl. Math.}, \textbf{69}(4):1162--1173, 2009.

\bibitem{N.Ben}
N.~Ben~Abdallah, F.~Castella, and F.~Méhats.
\newblock Time averaging for the strongly confined nonlinear {S}chr\"{o}dinger
  equation, using almost-periodicity.
\newblock {\em J. Differ. Equ.}, \textbf{245}:154–200, 2008.

\bibitem{N.Ben3}
N.~Ben~Abdallah, F.~Méhats, C.~Schmeiser, and R.~M. Weishäupl.
\newblock The nonlinear {S}chrödinger equation with a strongly anisotropic
  harmonic potential.
\newblock {\em SIAM J. Math. Anal.}, \textbf{37}(1):189--199, 2005.

\bibitem{H}
B.~Bongioanni and K.~M. Rogers.
\newblock Regularity of the {S}chr\"{o}dinger equation for the harmonic
  oscillator.
\newblock {\em Ark. Mat.}, \textbf{49}(2):217--238, 2011.

\bibitem{mana}
M.-R. Choi and Y.-R. Lee.
\newblock Averaging of dispersion managed nonlinear {S}chrödinger equations.
\newblock {\em Nonlinearity}, \textbf{35}(4):2121–2133, 2022.

\bibitem{FF}
F.~Delebecque-Fendt and F.~Méhats.
\newblock An effective mass theorem for the bidimensional electron gas in a
  strong magnetic field.
\newblock {\em Commun. Math. Phys.}, \textbf{292}(3):829--870, 2009.

\bibitem{LBL}
E.~Faou, P.~Germain, and Z.~Hani.
\newblock The weakly nonlinear large-box limit of the 2{D} cubic nonlinear
  {S}chrödinger equation.
\newblock {\em J. Am. Math. Soc.}, \textbf{29}(4):915--982, 2016.

\bibitem{RH}
J.~Fennell.
\newblock Resonant {H}amiltonian systems associated to the one-dimensional
  nonlinear {S}chrödinger equation with harmonic trapping.
\newblock {\em Commun. Partial Differ. Equ.}, \textbf{44}(12):1299--1344, 2019.

\bibitem{2}
R.~L. Frank, F.~Méhats, and C.~Sparber.
\newblock Averaging of nonlinear {S}chr\"{o}dinger equations with strong
  magnetic confinement.
\newblock {\em Commun. Math. Sci.}, \textbf{15}(7):1933--1945, 2017.

\bibitem{Fu79}
D.~Fujiwara.
\newblock A construction of the fundamental solution for the {S}chrödinger
  equation.
\newblock {\em J. Anal. Math.}, \textbf{35}:41--96, 1979.

\bibitem{Fu80}
D.~Fujiwara.
\newblock Remarks on convergence of the {F}eynman path integrals.
\newblock {\em Duke Math. J.}, \textbf{47}(3):559--600, 1980.

\bibitem{CR}
P.~Germain, Z.~Hani, and L.~Thomann.
\newblock On the continuous resonant equation for {{NLS. I.}} {D}eterministic
  analysis.
\newblock {\em J. Math. Pures Appl.(9)}, \textbf{105}(1):131--163, 2016.

\bibitem{BEC}
C.~Hainzl and B.~Schlein.
\newblock Dynamics of {B}ose-{E}instein condensates of fermion pairs in the low
  density limit of {BCS} theory.
\newblock {\em J. Funct. Anal.}, \textbf{265}(3):399--423, 2013.

\bibitem{MS}
Z.~Hani and L.~Thomann.
\newblock Asymptotic behavior of the nonlinear {S}chrödinger equation with
  harmonic trapping.
\newblock {\em Commun. Pure Appl. Math.}, \textbf{69}(9):1727--1776, 2016.

\bibitem{M}
F.~Méhats and C.~Sparber.
\newblock Dimension reduction for rotating {B}ose-{E}instein condensates with
  ainosotropic confinement.
\newblock {\em Discrete Contin. Dyn. Syst.}, \textbf{36}(9):5097--5118, 2016.

\bibitem{1}
Y.~Nakamura and A.~Shimomura.
\newblock Local well-posedness and smoothing effects of strong solutions for
  nonlinear {S}chr\"{o}dinger equations with potentials and magnetic fields.
\newblock {\em Hokkaido Math. J.}, \textbf{34}(1):37–63, 2005.

\bibitem{StH}
A.~Poiret.
\newblock Solutions globales pour l'\'{e}quation de {S}chr\"{o}dinger cubique
  en dimension 3.
\newblock {\em Preprint}, arXiv:1207. 1578 [math.AP], 2012.

\bibitem{iwa}
E.~Richman and C.~Sparber.
\newblock Strong magnetic field limit in a nonlinear {I}watsuka-type model.
\newblock {\em J. Differ. Equ.}, \textbf{302}:334--366, 2021.

\bibitem{Yajima}
K.~Yajima and G.~Zhang.
\newblock Local smoothing property and {S}trichartz inequality for
  {S}chrödinger equations with potentials superquadratic at infinity.
\newblock {\em J. Differ. Equ.}, \textbf{202}(1):81–110, 2004.

\end{thebibliography}
}

\end{document}